\documentclass[11pt]{article}

\usepackage{ulem}
\usepackage{amssymb,amsbsy,amsmath,amsfonts,amssymb,amscd,amsthm}
\usepackage{cancel}
\usepackage{xcolor}
\usepackage{bbm}
\usepackage{graphicx} 
\usepackage{multicol}
\graphicspath{{./}{immagini/}} 
\usepackage{tikz}
\usetikzlibrary{matrix}

\newcommand{\norm}[1]{\|#1 \|}

\newcommand{\R}{\mathbb{R}}

\newcommand{\N}{\mathbb{N}}
\newcommand{\T}{\mathbb{T}}

\newcommand{\mP}{\mathcal P}

\newcommand{\beq}{\begin{equation}}
\newcommand{\eeq}{\end{equation}}
\def\a{\alpha}

\def\d{\delta}
\def\D{\Delta}

\def\l{\lambda}
\def\L{\Lambda}
\def\m{\mu}
\def\n{\nu}
\def\r{\rho}

\def\o{\omega}

\def\G{\Gamma}

\def\th{\theta}

\def\e{\varepsilon}

\def\pd{\partial}
\def\half{\frac{1}{2}}

\newcommand{\cH}{{\cal H}}

\newcommand{\cK}{{\cal K}}

\newcommand{\cM}{{\cal M}}

\newcommand{\cP}{{\cal P}}

\newcommand{\Id}{\mathop{Id}}

\newcommand{\diver}{{\rm div}}

\newcommand{\bu}{\bar u}
\newcommand{\bm}{\bar m}

\DeclareMathOperator{\Wass}{\mathbf{d}_1}
\DeclareMathOperator{\Wassgen}{{\mathbf{d}}_{\mathcal{M}}}
\newtheorem{theorem}{Theorem}[section]
\newtheorem{lemma}[theorem]{Lemma}

\newtheorem{proposition}[theorem]{Proposition}

\newtheorem{remark}[theorem]{Remark}

\numberwithin{equation}{section}

\begin{document}

\title{On quasi-stationary  Mean Field Games of Controls}
 
\author{Fabio Camilli\footnotemark[1]  \and  Claudio Marchi\footnotemark[2]}

\date{version: \today}
\maketitle
\footnotetext[1]{Dip. di Scienze di Base e Applicate per l'Ingegneria,  Universit{\`a}  di Roma  ``La Sapienza", via Scarpa 16, 00161 Roma, Italy  ({\tt  fabio.camilli@uniroma1.it}).}
%%%%
\footnotetext[2]{Dip. di Ingegneria dell'Informazione \& Dip. di Matematica ``Tullio Levi-Civita'', Universit\`a di Padova, via Gradenigo 6/B, 35131 Padova, Italy ({\tt claudio.marchi@unipd.it}).}

\begin{abstract}
In Mean Field Games of Controls,  the dynamics of the single agent is influenced not only by the distribution of the agents, as in the classical theory, but also by the distribution of their optimal strategies. In this paper, we study quasi-stationary Mean Field Games of Controls, in which the strategy-choice mechanism of the agent is different from the classical case: the generic agent cannot predict the evolution of the population, but chooses its strategy only on the basis of the information available at the given instant of time, without anticipating. We prove existence and uniqueness for the solution of the corresponding quasi-stationary Mean Field Games system  under different sets of hypotheses and we provide some examples of models  which fall within these hypotheses.
\end{abstract}
	
\noindent
{\footnotesize \textbf{AMS-Subject Classification:} 35Q91, 49N70, 35B40, 35Q89}.\\
{\footnotesize \textbf{Keywords:} Mean field games, Quasi-stationary models, Nonlinear coupled PDE systems, Nash equilibria}.

\section{Introduction}
The Mean Field Games (MFG in short) theory concerns differential games for a large population  where the strategies of the agents are affected by the state distribution of the other players through a mean field effect. The corresponding MFG system, composed of a Hamilton-Jacobi-Bellman (HJB in short) equation and a Fokker-Planck (FP in short) equation, characterizes the Nash equilibrium for the population, i.e. the best strategy for the agents if all the others keep their choice. In this theory, the individual is assumed to be able to forecast the behaviour of the population at any later time, a somewhat restrictive assumption for some models such as pedestrian motion. In \cite{mouzouni}, it is  considered a different strategy mechanism: the agent assumes that the environment is immutable and, at each instant, it decides its behaviour only on the basis of the information available at the current time without anticipating the future. This leads to study a class of quasi-stationary MFG systems, where a stationary HJB  equation is coupled with an evolutive FP equation.\\
Recently, in \cite{Card_Lehalle,Carmona_Lacker,gomes_patrizi_voskanyan,Kobeissi1,Kobeissi2,lauriere_tangpi}, it has been introduced a significant generalization of the MFG model, called MFG of Controls, where the strategies of the agents depend not only on the position of other players but also on their strategy. The corresponding MFG system, with respect to the classical one, involves an additional fixed-point equation for the joint distribution of states and controls.\\
Aim of this paper it to extend the model developed in \cite{mouzouni} to Mean Field Games  of Controls. Hence, we deal with the following system
\begin{equation}
\label{QSS_discount} 
\left\{
\begin{aligned}
&- \D u + H(x,Du ; \m(t))+\r u =0 \quad &\mbox{ in } \T^d,\,\forall t\in [0,T]\\
&\partial_{t}m -  \D m -\diver(m  H_{p}(x,Du ; \m))=0&\mbox{ in } Q\\
&\m (t)=\left( \Id,\a^*(\cdot,Du(t);\m(t))\right) \sharp m (t)& \mbox{ in } [0,T] \\
& m (0)=m_{0}& \mbox{ in } \T^d, 
\end{aligned}
\right.
\end{equation}
and the corresponding ergodic version
\begin{equation}
\label{QSS_ergodic} 
\left\{
\begin{aligned}
&- \D u+ H(x,Du; \m(t))+\l=0 \quad &\mbox{ in } \T^d,\,\forall t\in [0,T]\\
&\partial_{t}m- \D m-\diver(m H_{p}(x,Du; \m))=0   &\mbox{ in } Q \\
&\m (t)=\left( \Id,\a^*(\cdot,Du(t);\m(t))\right) \sharp m(t) & \mbox{ in } [0,T] \\
&  u(0,t)=0 &  \mbox{ in } [0,T] \\
& m(0)=m_{0} & \mbox{ in } \T^d.
\end{aligned}
\right.
\end{equation} 
Here $\T^d$ denotes the unit torus in~$\R^d$, $Q=\T^d \times(0,T)$, $\rho>0$ and $f\#m$ stands for the push-forward of the measure $m$ by the function $f$; $Du$ and $\Delta u$ denote the gradient and respectively the Laplacian of the function~$u$ w.r.t. the variable~$x$ while $H_p$ is the gradient of $H=H(x,p;\mu)$ w.r.t.~$p$ (and we refer below for the precise definition of~$\a^*$). Moreover the function $m_{0}$, with $m_0\ge 0, \int_{\T^d} m_0(x)dx=1$, represents the initial distribution of players, $u$ is the value function of the single agent, $\lambda\in\R$ the ergodic cost and it can be interpreted as a long run average cost (see, for instance, \cite{AB}), $\mu$ the distribution  of the pair state-control and $m$ its first marginal (and amounts to the distribution of players). Denoted with $A$  the space of controls  and with $\mP(\T^d\times A)$  the space of probability measures on $\T^d\times A$, we consider the Hamiltonian $H:\T^d \times \R^d\times \mP(\T^d\times A)\to\R$ defined by
\begin{equation}\label{eq:hamiltonian}
H(x,p;\mu)=\sup_{a\in A}\left\{ -p\cdot b(x,a;\mu)-\ell(x,a;\mu)\right\}.
\end{equation}
Let us briefly recall how this operator arises in optimal control theory. For $\m(t)$ fixed, any single agent has a dynamic obeying to
\[
dx_s=b(x_s,a_s;\m(t)) dt + \sqrt{2} dW_s\qquad \textrm{in }(0,\infty),\qquad x_0=x
\]
and aims to choose an admissible control~$a_\cdot$ so to minimize the cost
\[
{\mathbb E}\int_0^\infty e^{-\rho s}\ell(x_s,a_s;\m(t))\, ds
\]
(here, $W$ is a $d$-dimensional Brownian motion and ${\mathbb E}$ denotes the expectation). It is well known (see \cite{AB} and references therein) that the value function associated to this optimal control problem solves the HJB equation in~\eqref{QSS_discount}.
In particular, the HJB equation in~\eqref{QSS_discount} is affected only by the value of~$\m$ at time~$t$; in other words, the single agent chooses its strategy ``freezing'' $\m$ at the present time~$t$ without forecasting the future evolution of~$\m$.\\
In fact, the first  equations in the  systems~\eqref{QSS_discount} and~\eqref{QSS_ergodic}  represent two families of stationary partial differential equations parametrized in the time variable, where the dependence on~$t$ is only through the measure $\mu$. In particular, these systems loss the standard forward-backward structure of MFG systems.\\
Let $\a^*=\a^*(x,p;\m)$ be a map (defined in assumption~$(H2)$ below) which associates to $(x,p)$ and $\m$ the corresponding optimal  control, i.e.
\begin{equation}\label{eq:optimal_control}
b(x,\a^*;\mu)=-H_p(x,p;\mu).
\end{equation}
In the equilibrium condition, at each instant~$t$, $m(t)$ is a measure on~$\T^d$ and $\mu(t)$ is the image of $m(t)$ by the map $x\mapsto (x,\a^*(x,Du(x,t);\mu(t)))$; this feature is expressed by the fixed-point relation given by the third equation in the  previous systems.\\
To show existence of a solution to \eqref{QSS_discount}, it is crucial  to have some regularity in time for the value function $u$. Actually, in the classical setting  of MFG theory, this regularity follows by the parabolicity of the HJB equation (for instance, see \cite{cardaliaguet_notes}). In the quasi-stationary setting of   \cite{mouzouni}, it is obtained  in two steps: first, the author obtains, via a continuous dependence estimate, a bound on the $C^2(\T^d)$-norm of $u(t)-u(s)$ by means of the Wasserstein distance $\Wass(m(t),m(s))$ (recall that in~\cite{mouzouni} the Hamiltonian depends on~$m$ and not on~$\mu$). Afterwards, using the FP equation, the latter quantity is bounded by $|t-s|^\half$ and  the regularity in time of $u$ follows. In the setting of the present paper, namely for quasi-stationary MFG of Controls, a similar argument fails. Indeed,  also in this case, we can apply some continuous dependence estimates and bound   the $C^2(\T^d)$-norm of $u(t)-u(s)$ by means of the Wasserstein distance $\Wass(\m(t),\m(s))$.  The aforementioned procedure fails when one tries to evaluate the latter quantity exploiting the equation for $\m$ because it obtains an estimate which depends on the $L^\infty$-norm of $Du(t)-Du(s)$, i.e. exactly the quantity we want to estimate. \\
To bypass the previous difficulty and recover some regularity property in time for the solution of the HJB equation, we  introduce two different settings:\par
1) In the first case, we use a continuous dependence estimate for the HJB equation and we assume that the constants that intervene in the regularity of the   Hamiltonian with respect to the joint distribution are small. Similar assumptions of  smallness also appear in other papers on MFG of Controls (see \cite{Bongini_Salvarani,Graber_Mayorga}). Note that this smallness assumption does not concern the length of the time interval. \par
2) Otherwise, we assume   a stability property of the Hamiltonian with respect to the joint distribution. This assumption allows us to obtain continuity in time of the gradient of the value function, uniformly with respect to $\mu$. A similar idea was exploited in \cite{Card_Lehalle}.\\
From a modeling point of view, the former setting corresponds to require that the mean field, given by the joint distribution, has a moderate influence on the behaviour of the single agent. In the latter case, we show that the assumption is satisfied if the agent, although it decides the strategy moment by moment, has also some knowledge of the past evolution of the joint measure (memory effect).
However, when the Hamiltonian has a separated dependence on the joint distribution, taking advantage of some specific properties of the HJB equation, we can drop both these assumptions.\\
In both settings, we first prove existence of a solution for system~\eqref{QSS_discount} and find several estimates; afterwards, exploiting such estimates, we obtain a solution to system~\eqref{QSS_ergodic} letting $\rho\to 0^+$. All our existence results for~\eqref{QSS_discount} (in Theorems~\ref{thm:existence_laxmilgram},~\ref{thm:1pt_fisso}, and Proposition~\ref{thm:separated} below) rely on Schauder fixed point theorem. Intuitively, one expects to look for a fixed point of a map which solves the three equations in~\eqref{QSS_discount} separately. This procedure needs the well posedness of the third equation alone which in turns is obtained solving another fixed point problem. Theorem~\ref{thm:existence_laxmilgram} follows this strategy which is similar to the ones in~\cite{Card_Lehalle,Kobeissi1}. On the other hand, in Theorem~\ref{thm:1pt_fisso} and in Proposition~\ref{thm:separated}  we adopt a different approach looking for a fixed point of a unique map and without solving separately the third equation in~\eqref{QSS_discount}.

We also establish uniqueness of the solution under the assumption of smallness of constants. The proof relies on continuous dependence estimates of $\mu$ w.r.t. the data of the third equation. Unfortunately, in the second setting, we do not have such a regularity so this uniqueness result cannot be extended.

This paper is organized as follows: Section~\ref{sec:hyp_prelim} contains notations, assumptions and several useful results on the first two equations in systems~\eqref{QSS_discount} and~\eqref{QSS_ergodic}. In Section~\ref{sec:exist_via_Cont_dep} we obtain existence and uniqueness of the solutions to~\eqref{QSS_discount} and to~\eqref{QSS_ergodic} under an assumption on the smallness of the constants. Section~\ref{sec:exis_via_reg_dep} is devoted to obtain existence of solutions to our problems under a stability property of the Hamiltonian~$H$ w.r.t.~$\mu$. Section~\ref{sec:comments_examples} provides some examples where our assumptions are satisfied; it also contains the particular case where $H$ depends separately on~$\mu$. Finally, Appendix~\ref{sec:appendix} contains the proofs of several technical results.

%%%%%%%%%%%%%%%%%%%%%%%%%%%%%%%%%%%%%%%%%
%  Assumptions and preliminary results  %
%%%%%%%%%%%%%%%%%%%%%%%%%%%%%%%%%%%%%%%%%
\section{Assumptions and preliminary results}\label{sec:hyp_prelim}
In this section we will introduce the assumptions and we will discuss some preliminary results we need  for the study of systems~\eqref{QSS_discount} and~\eqref{QSS_ergodic}. \\
We denote with $\cP(X)$ the set of probability measures on the compact  separable metric space $X$ and we recall that  $\cP(X)$ is a compact topological space when endowed with the weak$^*$-convergence. Moreover the topology  on  $\cP(X)$ is metrizable by the Kantorovich-Rubinstein distance, defined  by
\[
\Wass(\m,\m')=\sup\big\{\int_X f(x)d(\m-\m'):\,f:X\to\R\,\mbox{is $1$-Lipschitz continuous}\big\}.
\]
We denote with $L^p(\T^d)$, $1\le p\le\infty$, the set of $p$-summable periodic functions and with $W^{k,p}(\T^d)$, $k\in\N$ and $1\le p\le\infty$, the Sobolev space of periodic functions having $p$-summable weak derivatives up to order $k$ and we set $H^k(\T^d)=W^{k,2}(\T^d)$. All these spaces are endowed with the corresponding standard norms.
For $\theta \in(0,1]$, we use the $\theta$-H\"older seminorm
\[
[f]_{\theta,A}:= \sup\left\{\frac{|f(x)-f(y)|}{|x-y|^\theta};\, x,y\in A,\, x\ne y\right\}.
\]
We denote with $C^1(\T^d)$ (respectively, $C^{1,\th}(\T^d)$ with $\th\in(0,1]$) the space of functions~$f$, defined on~$\T^d$ with continuous first order derivatives, such that the norm
\[
\|f\|_{C^1(\T^d)}:=\|f\|_{L^\infty(\T^d)}+\|Df\|_{L^\infty(\T^d)}
\]
(respectively, $\|f\|_{C^{1,\th}(\T^d)}:=\|f\|_{C^1(\T^d)}+ [Df]_{\theta,\T^d}$) is finite. In a similar way, we denote the spaces $C^2(\T^d)$ and $C^{2,\th}(\T^d)$.
For any $\d\in(0,1]$, we denote with $C^{\d,\d/2}(Q)$ the space of functions~$m$ on~$Q$ such that
\begin{equation*}
\sup_{\scriptsize{ \begin{array}{c}(x,t),(x',t')\in Q\\ (x,t)\ne (x',t')\end{array}}}\frac{|m(x,t)-m(x',t')|}{|x-x'|^\d+|t-t'|^{\d/2}}<\infty.
\end{equation*}
Defined $W^{1,0}_s(Q)$ as the space of functions on~$Q$ such that the   norm
\[
\norm{u}_{W^{1,0}_s(Q)}:=\norm{u}_{L^s(Q)}+ \norm{D u}_{L^s(Q)}
\]
is finite, we denote with $\cH_s^{1}(Q)$ the space of   functions $u\in W^{1,0}_s(Q)$ with $\partial_t u\in (W^{1,0}_{s'}(Q))'$, equipped with the natural norm
\begin{equation*}
\norm{u}_{\mathcal{H}_s^{1}(Q)}:=\norm{u}_{W^{1,0}_s(Q)}
+\norm{\partial_tu}_{(W^{1,0}_{s'}(Q))'}\ .
\end{equation*}
Let us now recall same useful properties of the spaces~$\cH^1_s(Q)$; for the proof we refer to \cite[Theorem XVIII.2.1]{DL} and \cite[Prop.2.12]{CG19} (see also \cite[Appendix A]{metafune_altri}).
\begin{lemma}\label{lemma:H1s}
The space ~$\cH^1_2(Q)$ is continuously embedded in $C([0,T];L^2(\T^d))$: there exists a constant~$c_\cH$ such that, for every $m\in \cH^1_2(Q)$, there holds
\[
\|m(t)\|_{L^2(\T^d)}\leq c_\cH \|m\|_{\cH^1_2(Q)} \qquad \forall t\in[0,T].
\]
Moreover, for $s>d+2$, $\cH_s^1(Q)$ is continuously embedded onto $C^{\delta,\delta/2}(Q)$ for some $\delta\in (0,1)$.
\end{lemma}

We consider a Hamiltonian $H$  given as in \eqref{eq:hamiltonian} and the following assumptions:
\begin{itemize}
%%%%%  Assumptions for existence %%%%%%%%%%%%%%%%%%%%%%%%%
\item[(H1)] 
The functions $b:\T^d\times A\times\cP(\T^d\times A)\to\R^d$ and $\ell:\T^d\times A\times\cP(\T^d\times A)\to\R$ are continuous and  the control set    $A$ is a compact, separable metric space (for simplicity we assume that~$A$ is a subset of some Euclidean space). Moreover, there exist two constants $K$ and $L$ such that
\begin{equation}\label{hyp_b}
\begin{split}
&|b(x,a;\nu)|\leq K,\qquad |b(x_1,a;\nu)-b(x_2,a;\nu)|\leq L|x_1-x_2|\\
&|\ell(x,a;\nu)|\leq K,\qquad |\ell(x_1,a;\nu)-\ell(x_2,a;\nu)|\leq L|x_1-x_2|
\end{split}
\end{equation}
for all $x,x_1,x_2\in \T^d$, $a\in A$, $\nu\in \mP(\T^d\times A)$.
We assume that $D_{x,p}H$ exists and is locally $\theta$-H\"older continuous for some $\theta\in(0,1]$.
\item[(H2)]  For any $(x,p,\n)\in\T^d\times\R^d\times \cP(\T^d\times A)$, there exists a unique $\a^*=\a^*(x,p;\n)\in A$ such that 
\[H(x,p;\nu)=  -p\cdot b(x,\a^*;\nu)-\ell(x,\a^*;\n). \]	
 Moreover the map $\a^*$ is continuous with respect to its arguments.
\item[(H3)] For each positive constant $R$, there exists a constant~$L_{\mu,R}$ such that
\begin{align*}
&\max_{x,a}|(b(x,a;\n_1)-b(x,a;\n_2))p|
+\max_{x,a}|\ell(x,a;\n_1)-\ell(x,a;\n_2)|\\
&\quad\le L_{\mu,R}\Wass(\n_1,\n_2)
\end{align*}
and
\[
[H(\cdot,\cdot,\n_1)-H(\cdot,\cdot,\n_2)]_{1, \T^d\times B(0,R)}\leq L_{\mu,R}\Wass(\n_1,\n_2)
\]
for any $p\in B(0,R)$ and $\n_1,\n_2\in \cP(\T^d\times A)$.\\
We shall denote $L_\mu$ the constant $L_{\mu,\bar K}$  where $\bar K$ is the constant introduced in Lemma~\ref{lemma:bounds_discount_ergodic} below.
\end{itemize}
%%%%%%%
\begin{remark}
Note that $\T^d\times A$ is compact and, consequently, by Prokhorov's theorem, also $\cP(\T^d\times A)$ is compact; hence, $b$ and $\ell$ are both uniformly continuous.
\end{remark}
\begin{remark}\label{rmk:H1+H3}
Assumptions $(H1)$ and $(H3)$ entail
\[
|H(x,p_1;\n_1)-H(x,p_2;\n_2)|\le K|p_1-p_2|+L_\mu\Wass(\n_1,\n_2)
\]
for any $x\in\T^d$, $p_1,p_2\in B(0,\bar K)$ with $\bar K$ as in \eqref{eq:bound_2}, $\n_1,\n_2\in \cP(\T^d\times A)$.
\end{remark}

%%%%%%%%%% Equazione di HJB %%%%%%%%%%%%%%%%%%%%%%
The following lemma is a classical result concerning existence, uniqueness and regularity of classical solutions to the  HJB equations in systems \eqref{QSS_discount} and \eqref{QSS_ergodic} (see for example  \cite[Thm. 4.1]{AB}, \cite[Thm. 5]{gomes} and \cite[Prop. 2.1]{marchi}).
\begin{lemma}\label{lemma:bounds_discount_ergodic}
For a fixed measure $\nu\in \mP(\T^d\times A)$, there exists a unique classical bounded solution $u^\r$ to the equation
\begin{equation}\label{eq:HJB_bound_discount}
- \D u + H(x,Du ; \nu)+\r u =0\quad \mbox{ in } \T^d.
\end{equation}
Moreover,
\begin{itemize}
\item[(i)] there exist a positive constant  $C_1$ and $\th\in(0,1)$, both independent of $\r$  and~$\nu$,  such that
\begin{align}
&\|\r u^\r\|_{L^\infty} \le K,\label{eq:bound_1}\\
&\|u^\r-u^\r(0)\|_{C^{2,\theta}(\T^d)} \le C_1(1+K+L)=:\bar K,\label{eq:bound_2}
\end{align}
where $K$, $L$ as in \eqref{hyp_b} (recall that $K$, $L$ are independent of~$\nu$);
\item[(ii)] for $\r\to 0^+$, $\r u^\r\to \l$, $u^\r-u^\r(0)\to u$ and the couple $(u,\l)$ is the unique solution to
\begin{equation}\label{eq:HJB_bound_ergodic}
- \D u + H(x,Du ; \nu)+\l =0\quad \mbox{ in } \T^d,\qquad u(0)=0. 
\end{equation}
Moreover
\begin{equation}\label{eq:bound_3}
\|u\|_{C^{2,\theta}(\T^d)} \le \bar K.
\end{equation}
\end{itemize}
\end{lemma}

%%%%%%%%% continuous dependence estimate %%%%%%%%%%

We now give a continuous dependence estimate for the solution of \eqref{eq:HJB_bound_discount} and \eqref{eq:HJB_bound_ergodic} with respect to the data of the problem.
\begin{lemma}\label{lemma:dip_continua}
Given $\n_1,\n_2\in\cP(\T^d\times A)$, denote with $u^\r_1$, $u^\r_2$ the  solutions of \eqref{eq:HJB_bound_discount} with $\n$ replaced respectively with $\n_1$ and $\n_2$ and set $w^\r_i:=u^\r_i-u^\r_i(0)$. Then, there exists a positive constant $C_0$, independent of $\r$, $\n_1$, $\n_2$, such that
\begin{multline} \label{eq:dc_0}
\|w^\r_1-w^\r_2\|_{C^2(\T^d)}\le  [H(\cdot,\cdot,\n_1)-H(\cdot,\cdot,\n_2)]_{1, \T^d\times B(0,\bar K)}+\\
C_0\Big( \max_{x,a}|b(x,a;\n_1)-b(x,a;\n_2)|+\max_{x,a}|\ell(x,a;\n_1)-\ell(x,a;\n_2)|\Big).
\end{multline}
Estimate \eqref{eq:dc_0} also holds for $u_1, u_2$  solutions to \eqref{eq:HJB_bound_ergodic} corresponding to $\n_1$ and respectively $\n_2$.
\end{lemma}
The proof is postponed to the Appendix.
%%%%%%%%% Equazione di FP %%%%%%%%%%%%%%%%%%%%%%%%%
We now study the Fokker-Planck equation.
\begin{lemma}\label{lemma:linear_FP}
Given a bounded vector field $g:Q\to\R^d$  and  $m_0\in L^2(\T^d)$, $m_0\geq0$, then the problem
\begin{align*} 
\left\{
\begin{array}{ll}
\pd_t m-  \Delta m-\diver (g(x,t)m)=0&\mbox{ in }Q,\\
m(x,0)=m_0(x)&\mbox{ in }\T^d,
\end{array}
\right.
\end{align*}
has a unique  non negative  solution $m\in \cH_2^1(Q)$. Furthermore, we have 
\begin{align}
&\Wass(m(t),m(s))\le c_0|t-s|^\half,\qquad &t,s\in [0,T],\label{eq:holder_Wass}
\end{align}
with $ c_0=c_0(\|g\|_{L^\infty(Q;\R^d)},m_0)$. 
Moreover, if $m_0\in W^{1,s}(\T^d)$, $s\in (1,\infty)$, we also have
\begin{equation}\label{eq:bound_m}
\|m\|_{\cH_s^1(Q)}\le c_1
\end{equation}
for some constant  $ c_1=c_1(\|g\|_{L^\infty(Q;\R^d)},\|m_0\|_{W^{1,s}(\T^d)})$.
\end{lemma}
For the proof of~\eqref{eq:bound_m} we refer for example to \cite[Lemma 2.1]{camilli_e_altri}, while for the proof of~\eqref{eq:holder_Wass} it is enough to apply the results in~\cite[Lemma 3.4 and 3.5]{cardaliaguet_notes} to Lipschitz regularizations $g_n$ of $g$  and to pass to the limit (indeed, in these lemmas, the constants only depend on $\|g_n\|_{L^\infty}$).
\begin{remark}
The compactness of~$\T^d$ entails
\begin{equation}
 \int_{\T^d}|x|^2dm(t)\le 1\qquad \forall t\in [0,T].\label{eq:moment_m}
\end{equation}
\end{remark}

%%%%%%%%%%%%%%%%%%%%%%%%%%%%%%%%%%%%%%%%%%%%%%%%%%%%%
%  well posedness via  Continuous dependance        %
%%%%%%%%%%%%%%%%%%%%%%%%%%%%%%%%%%%%%%%%%%%%%%%%%%%%%
\section{Well-posedness via continuous dependence estimates}\label{sec:exist_via_Cont_dep}
In this section, we   prove the well-posedness of systems \eqref{QSS_discount} and   \eqref{QSS_ergodic}  making, besides $(H1)$, $(H2)$ and $(H3)$  the following additional assumption
\begin{itemize}
\item[(H4)] 
There exist $\l_0\in [0,1)$ and $\l_1\in [0,\infty)$ such that
\begin{equation}\label{hyp:key_assumption2}
\begin{split}
|\a^*(x_1,p_1;\n_1)- \a^*(x_2,p_2;\n_2)|\le& \l_0\Wass(\n_1,\n_2)\\
&+\l_1(|x_1-x_2|+|p_1-p_2|)
\end{split}
\end{equation}
for all $x_1,x_2\in \T^d$ $p_1,p_2\in \R^d$, $\n_1,\n_2\in   \cP(\T^d\times A)$.
\end{itemize}
 The assumption $(H4)$ is reminiscent of a similar assumption in the paper~\cite{Kobeissi1} which copes evolutive systems of MFG of controls.
%%%%%%%%%%%%%%% Well-posedness fixed point equation %%%%%%%%%%%%%%%%%%%%%%%%
We first prove the well-posedness and some properties of the fixed point equation in \eqref{QSS_discount}
(for the proof, see the Appendix).
%%%%%%%%%%
\begin{lemma}\label{lemma:fixed_point}
Assume $(H2)$ and $(H4)$. Given $m\in \cP(\T^d)$ and a measurable function $p:\T^d\to\R^d$, there exists a unique solution $\mu\in\cP(\T^d\times A)$ of the fixed point equation
\begin{equation}\label{eq:fixed_1}
\m =\left( \Id,\a^*(\cdot,p(\cdot);\m)\right) \sharp m.
\end{equation}
Moreover the following properties hold:
\begin{itemize}
\item[(i)] 
Given $m\in L^2(\T^d)$, $p_n, p\in  L^2(\T^d;\R^d)$, let $\m_n,\m\in \cP(\T^d\times A)$ be the corresponding solutions to \eqref{eq:fixed_1}. Then, the following estimate holds
\begin{equation*}
\Wass(\m_n,\m)\leq \frac{\l_1}{1-\l_0}\|p_n-p\|_{L^2}\|m\|_{L^2}.
\end{equation*}	
\item[(ii)] Given $m_n, m \in\cP(\T^d)$ and  $p\in W^{1,\infty}(\T^d;\R^d)$   with Lipschitz constant $L_p$, 
 let $\m_n,\m\in \cP(\T^d\times A)$ be the corresponding solutions to \eqref{eq:fixed_1}. Then, the following estimate holds
\[
\Wass(\m_n,\m)\leq \frac{1+\l_1(1+ L_p)}{1-\l_0}\Wass (m_n,m).
\]
\end{itemize}
\end{lemma}

\begin{remark}\label{remark:3.3}
Since $A$ is bounded, there exists a constant $c_2$  such that 
\begin{equation*}
\int_{\T^d\times A}(|x|^2+|a|^2)d\m(x,a)\le c_2\qquad\forall \mu\in\cP(\T^d\times A).
\end{equation*}
\end{remark}
In the following theorem, we   prove    existence of a solution to \eqref{QSS_discount}  under a smallness assumption on the constants that intervene in the regularity of~$H$ w.r.t.~$\m$.
\begin{theorem}\label{thm:existence_laxmilgram}
Assume $(H1)$, $(H2)$, $(H3)$, $(H4)$, $m_0\in H^1(\T^d)$ and
\begin{equation}\label{hyp:smallness_Lipschitz}
\frac{\l_1c_1c_\cH (C_0+1)}{1-\l_0}L_\mu< 1
\end{equation}
where $C_0$ as in \eqref{eq:dc_0}, $\l_0,\l_1$ as in \eqref{hyp:key_assumption2}, $c_1$ as in \eqref{eq:bound_m}$, c_\cH$ as in Lemma~\ref{lemma:H1s} and $L_\mu$ as in (H3). Then, problem \eqref{QSS_discount} admits a   solution $(u,m,\m)$, where
$u\in C([0,T], C^2(\T^d))$ is a classical solution of the  HJB equation for any $t\in [0,T]$, $m\in \cH^1_2(Q)$  is a weak solution of the Fokker-Planck equation  and  $\mu\in C([0,T],\cP(\T^d\times A))$ satisfies the fixed-point equation for any $t\in [0,T]$. 
Moreover, there exists a positive constant $L$,  
independent of $\rho$, such that: for any $t,s\in[0,T]$ there hold:
\begin{equation}\label{eq:regularity}
\begin{array}{ll}
(i)& \Wass(m(t),m(s))\leq L|t-s|^\half, \quad  \|m\|_{\cH^1_2}\leq L,\\
(ii)& \Wass(\m(t),\m(s))\leq L|t-s|^\half,\\
(iii)&\|[u(\cdot,t)-u(0,t)]-[u(\cdot,s)-u(0,s)]\|_{C^2(\T^d)} \leq L|t-s|^\half,\\
(iv) & \|\r u\|_{L^\infty}\leq L,\quad \|u(\cdot,t)-u(0,t)\|_{C^2(\T^d)}\leq L,\\
(v) & \textrm{if moreover $m_0\in W^{1,s}(\T^d)$ with $s>2$, then $\|m\|_{\cH^1_s}\leq L$}.
\end{array}
\end{equation}
\end{theorem}
%%%%%%%%%%%
\begin{proof}
For $(\bu,\bm)\in C([0,T], C^{1,\th}(\T^d))\times C([0,T], \cP(\T^d))$, with $\bm\in \cH^1_2(Q)$  and $\|D\bu(t,\cdot)\|_{C^{1,\th}}\leq \bar K$ for any $t\in[0,T]$, where $\bar K$ and $\th$ are given in Lemma~\ref{lemma:bounds_discount_ergodic}, consider the  map  $(u,m)=\G(\bu,\bm)$ defined as follows:\\
\textit{(i)} To each $(\bu,\bm)$, we associate the unique map $\m:[0,T]\to \cP(\T^d\times A)$ solution to relation
\begin{equation*}
\m(t) =\left( \Id,\a^*(\cdot,D\bu(t);\m(t))\right) \sharp \bm(t),\qquad t\in [0,T];
\end{equation*}
this definition is well posed by Lemma \ref{lemma:fixed_point} and because $\bm(t)\in L^2(\T^d)$ by Lemma~\ref{lemma:H1s}.
Moreover, we observe that Lemma \ref{lemma:fixed_point} (i) and (ii) ensure
\begin{equation}\label{eq:stima_mu}
\begin{split}
\Wass(\m(t),\m(s))&\le\frac{\l_1}{1-\l_0}\|D\bu(t)-D\bu(s)\|_{L^\infty}c_{\cH}\|\bar m\|_{\cH_2^1} \\
&+\frac{1+\l_1(1+ \bar K)}{1-\l_0}\Wass(\bm(t), \bm(s)).
\end{split}
\end{equation}
Hence, by Lemma~\ref{lemma:H1s}, $\mu$ belongs to $C([0,T],\cP(\T^d\times A))$.\\
%%%%%%%%
\textit{(ii)} To $\m$ in the previous step, we associate the unique solution $u:[0,T]\to C^{2,\th}(\T^d)$  to the equation
\[ \r u(t)-\Delta u(t) +H(x,D u(t);\m(t))=0,\qquad x\in\T^d.\]
We claim: $Du\in C([0,T],C^{1}(\T^d))$. Actually, applying Lemma~\ref{lemma:dip_continua} with $\nu_1=\mu(t)$ and $\nu_2=\mu(s)$, by (H3) and~\eqref{eq:stima_mu}, we get
\begin{eqnarray*}
\|Du(t)-Du(s)\|_{C^{1}}&\leq&  (C_0+1) L_\mu\Wass(\mu(t),\mu(s))\\
&\leq&  (C_0+1) L_\mu \left[\frac{\l_1}{1-\l_0}\|D\bu(t)-D\bu(s)\|_{L^\infty}c_{\cH}\|\bm\|_{\cH^1_2}\right.\\
&&\qquad\left.+\frac{1+\l_1(1+ \bar K)}{1-\l_0}\Wass(\bm(t), \bm(s))\right]\\
& \to& 0 \qquad \textrm{as }s\to t.
\end{eqnarray*}
%%%%%
\textit{(iii)} Finally, given $u$ and $\mu$ as in the previous steps, let $m \in C([0,T],\cP(\T^d))$ be the unique solution to the FP equation
\[
\partial_{t}m- \D m-\diver(m H_{p}(x,Du; \m(t)))=0 \quad\textrm{in } Q,\qquad m(0)=m_0 \quad\textrm{on } \T^d
\]
obtained in Lemma \ref{lemma:linear_FP} with
\[g(x,t)=H_{p}(x,Du(x,t); \m(t))=-b(x,\a^*(x,Du(x,t);\m(t)),\m(t))\]
(and taking into account relation \eqref{eq:optimal_control}).\par
We shall obtain the existence of a solution to \eqref{QSS_discount} proving, by Schauder fixed point  theorem, that the map $\G$ has a fixed point in the set $\cK$ defined by
\begin{equation*} 
\cK=\left\{
\begin{array}{l}
\qquad (\bu,\bm)\in C([0,T], C^{1,\th}(\T^d))\times C([0,T], \cP(\T^d))\,\text{s.t.}\\[4pt]
(\cK1)\, \r\|\bu(t)\|_{L^\infty}\leq K,\, \|D\bu(t)\|_{C^{1,\theta}(\T^d)}\leq \bar K\\
(\cK2)\, \Wass(\bm(t),\bm(s))\le c_0|t-s|^\half,\,\|\bm\|_{\cH^1_2(Q)}\le c_1\\
(\cK3)\,\|D\bu(t)-D\bu(s)\|_{L^\infty}\le C^*|t-s|^\half\\
(\cK4)\,  \r\|\bu(t)-\bu(s)\|_{L^\infty}\le C^*_1|t-s|^\half \end{array}
\right\},
\end{equation*}
endowed with the $C([0,T], C^{1,\th}(\T^d))\times C([0,T], \cP(\T^d))$ topology, where $K$ as in~\eqref{eq:bound_1}, $\theta$ and $\bar K$ as in \eqref{eq:bound_2}, $c_0$, $c_1$ as in \eqref{eq:holder_Wass}, \eqref{eq:bound_m}, and $C^*$ and $C^*_1$ are two constants that will be suitably chosen later on in~\eqref{eq:Lstar} and respectively in~\eqref{C*1} (and they will be independent of~$\rho$).

We claim that $\cK$ is a nonempty, convex and compact set.  We prove only compactness, since the other two properties are obvious. Consider a sequence $\{(\bu_n,\bm_n)\}$ of elements in $\cK$;  we want to prove that, possibly passing to a subsequence, it converges to some $(\bu,\bm)\in \cK$. By properties $(\cK2)$ and Lemma~\ref{lemma:H1s}, Ascoli-Arzela theorem ensures that $\{m_n\}$ converges to some element $m\in C([0,T],\cP(\T^d))$ verifying the first estimate of~$(\cK2)$. Moreover, still by $(\cK2)$, $\{m_n\}$ also converges in the weak topology of~$\cH^1_2(Q)$ to some $m'$ which must coincide with $m$; hence $m$ fulfills $(\cK2)$.\\
On the other hand, we first note that the set $\left\{u\in C^{1,\th}(\T^d): \textrm{$u$ fulfills $(\cK1)$}\right\}$ is compact in the $C^{1,\th}$ topology. Moreover, by $(\cK1)$, the $\bu_n$'s are bounded in $C([0,T];C^{1,\th}(\T^d))$ uniformly in~$n$. By $(\cK3)$-$(\cK4)$ and  Ascoli-Arzela theorem, $\{\bu_n\}$ converges to some $u$ in the $C([0,T], C^{1,\th}(\T^d))$-topology and $u$ fulfills $(\cK3)$ and $(\cK4)$. In particular, for each $t$, $\{\bu_n(t)\}$ converges to $u(t)$ in the $C^{1,\th}$-topology. Again by Ascoli-Arzela theorem, we infer that $u(t)$ fulfills $(\cK1)$.

We show that $\G$ maps $\cK$ into itself.
For $(u,m)=\G(\bu,\bm)$, $(\cK1)$ and $(\cK2)$ follow  from Lemma \ref{lemma:bounds_discount_ergodic} and, respectively, from Lemma \ref{lemma:linear_FP}.
We show that also $(\cK3)$ holds. For any $s,t\in[0,T]$, by Lemma~\ref{lemma:dip_continua} and $(H3)$, we have
\begin{equation*}
\|Du(s)-Du(t)\|_{C^1(\T^d)} \leq (C_0+1) L_\mu \Wass(\mu(s),\mu(t)).
\end{equation*}
Estimate~\eqref{eq:stima_mu}, Lemma~\ref{lemma:H1s} and $(\cK2)$ yield
\begin{equation*}
\begin{split}
&\|Du(s)-Du(t)\|_{C^1}\\&\qquad \leq (C_0+1) L_\mu\left[\frac{\l_1c_1c_{\cH}}{1-\l_0}\|D\bu(t)-D\bu(s)\|_{L^\infty} +\frac{1+\l_1(1+ \bar K)}{1-\l_0}c_0|t-s|^\half\right]\\
&\qquad \leq \left[\frac{\l_1c_1c_{\cH}(C_0+1) L_\mu}{1-\l_0}C^* +(C_0+1) c_0L_\mu\frac{1+\l_1(1+ \bar K)}{1-\l_0}\right]|t-s|^\half
\end{split}
\end{equation*}
where the last inequality is due to $(\cK3)$. By~\eqref{hyp:smallness_Lipschitz}, for $C^*$ so large to fulfill
\begin{equation}\label{eq:Lstar}
C^*\geq \left[1-\frac{\l_1c_1c_\cH (C_0+1)}{1-\l_0}L_\mu\right]^{-1}(C_0+1) c_0L_\mu\frac{1+\l_1(1+ \bar K)}{1-\l_0},
\end{equation}
there holds 
\begin{equation*}
\|Du(s)-Du(t)\|_{C^1}\leq C^* |t-s|^\half,
\end{equation*}
namely $u$ fulfills $(\cK3)$.

Let us now prove that $(\cK4)$ holds. Actually, the comparison principle and Remark~\ref{rmk:H1+H3} ensure
\begin{eqnarray}\notag
\r\|u(t)-u(s)\|_{L^\infty}&\leq & \max\limits_{x\in\T^d}|H(x,Du(t);\mu(t))-H(x,Du(s);\mu(s))|\\  \label{13luglio}
&\leq & K\|Du(t)-Du(s)\|_{L^\infty}+L_\mu\Wass(\mu(t),\mu(s)).
\end{eqnarray}
By \eqref{eq:stima_mu}, Lemma~\ref{lemma:H1s} and $(\cK2)$-$(\cK3)$, we infer
\begin{equation*}
\r\|u(t)-u(s)\|_{L^\infty}\leq \left[KC^*+\frac{L_\mu\l_1c_1c_{\cH}C^*}{1-\l_0} +\frac{1+\l_1(1+ \bar K)}{1-\l_0}c_0L_\mu\right]|t-s|^\half;
\end{equation*}
choosing
\begin{equation}\label{C*1}
C^*_1\geq KC^*+\frac{L_\mu\l_1c_1c_{\cH}C^*}{1-\l_0} +\frac{1+\l_1(1+ \bar K)}{1-\l_0}c_0L_\mu,
\end{equation}
we obtain that $(\cK4)$ holds.
We conclude that $\G$ maps $\cK$ into itself.

It remains to prove that $\G$ is a continuous map in the  $C([0,T], C^{1,\th}(\T^d))\times C([0,T], \cP(\T^d))$ topology. Consider a sequence $\{(\bu_n,\bm_n)\}$ in $\cK$ converging to $(\bu,\bm)\in\cK$. Set $(u_n,m_n)=\G(\bu_n,\bm_n)$, $(u,m)=\G(\bu,\bm)$ and let $\m_n$, $\m$ be the solution of
\[
\m_n =\left( \Id,\a^*(\cdot,D\bu_n(t);\m_n(t))\right) \sharp \bm_n(t)
\]
and, respectively, of the corresponding equation with $(\bu,\bm)$ in place of $(\bu_n,\bm_n)$.
Denote with $\tilde \m_n$ the solution of
\[
\tilde \m_n =\left( \Id,\a^*(\cdot,D\bu(t);\tilde \m_n(t))\right) \sharp \bm_n(t).
\]
By Lemma \ref{lemma:fixed_point}(i) and (ii), Lemma~\ref{lemma:H1s} and $(\cK2)$, we have
\begin{align*}
\Wass(\m_n(t),\m(t))\le	\Wass(\m_n(t),\tilde\m_n(t))+\Wass(\tilde\m_n(t),\m(t))\\
\le \frac{\l_1c_1c_\cH}{1-\l_0}\|D\bu_n(t)-D\bu(t)\|_{L^\infty}+\frac{1+\l_1(1+\bar K)}{1-\l_0}\Wass(\bm_n(t),\bm(t))
\end{align*}
and therefore 
\begin{equation}\label{eq:conv_1}
\Wass(\m_n(t),\m(t))\to 0\quad  \text{for $n\to \infty$, uniformly in $t\in [0,T]$.}
\end{equation}
By Lemma \ref{lemma:dip_continua} and $(H3)$, we deduce
\[
\|[u_n(\cdot,t)-u_n(0,t)]-[u(\cdot,t)-u(0,t)]\|_{C^2(\T^d)}=o_n(1)
\]
where, since now on, $o_n(1)$ denotes a function which may change from line to line and such that $\lim_no_n(1)=0$, uniformly in $t$; hence we have
\begin{equation}\label{310bis}
\|Du_n(\cdot,t)-Du(\cdot,t)\|_{C^1}=o_n(1).
\end{equation}
Let us now prove that $u_n\to u$ in the topology of $C([0,T],C^{1,\th}(\T^d))$. To this end, by~\eqref{310bis}, it suffices to establish
\[
\|u_n(\cdot,t)-u(\cdot,t)\|_{L^\infty}=o_n(1).
\]
Arguing as in~\eqref{13luglio}, the comparison principle and $(H3)$ entail
\begin{eqnarray*}
\r\|u_n(t)-u(t)\|_{L^\infty}&\leq & \max\limits_{x\in\T^d}|H(x,Du_n(t);\mu_n(t))-H(x,Du(t);\mu(t))|\\
&\leq & K\|Du_n(t)-Du(t)\|_{L^\infty}+L_\mu\Wass(\mu_n(t),\mu(t)).
\end{eqnarray*}
By \eqref{eq:conv_1} and \eqref{310bis}, we get
\begin{equation*}
\r\|u_n(t)-u(t)\|_{L^\infty}=o_n(1)
\end{equation*}
which is equivalent to our claim. Therefore, $u_n\to u$ in the desired topology.\\
Let us now cope the convergence of the $m_n$'s. By \eqref{eq:optimal_control}, we get
\begin{align*}
\|H_p&(\cdot,Du_n(\cdot,t);\m_n(t))-H_p(\cdot,Du(\cdot,t);\m(t))\|_{L^\infty}\\
&\leq \omega\left(\|\a^*(\cdot,Du_n(\cdot,t);\m_n(t))-\a^*(\cdot,Du(\cdot,t);\m(t))\|_{L^\infty}+\Wass(\m_n(t),\m(t))\right)
\end{align*}
where $\omega$ is a modulus of continuity of $b$. By $(H4)$, we infer
\begin{equation}\label{310ter}
\begin{array}{l}
\|H_p(\cdot,Du_n(\cdot,t);\m_n(t))-H_p(\cdot,Du(\cdot,t);\m(t))\|_{L^\infty}\\
\qquad \le \omega \left((\l_0+1)\Wass(\m_n(t),\m(t))+\l_1 \|Du_n(\cdot,t)-Du(\cdot,t)\|_{L^\infty}\right)\\
\qquad \le \omega(o_n(1)+o_n(1))=o_n(1)
\end{array}
\end{equation}
where the last inequality is due to \eqref{eq:conv_1} and to \eqref{310bis}.

By $(\cK2)$ and Ascoli-Arzela theorem, there exists a subsequence $\{m_{n_k}\}$ converging to some function $\tilde m\in C([0,T],\cP(\T^d))$. Let us prove that $\tilde m=m$. Indeed, $m_{n_k}$ solves
\begin{align*}
\iint_{\T^d\times [0,T]}(-\partial_t \psi-\Delta\psi+H_p(x,Du_{n_k};\m_{n_k})\cdot D\psi)m_{n_k}dxdt+\int_{\T^d}\psi m_0dx=0
\end{align*}
for any $\psi \in C^\infty_0(\T^d\times[0,T))$. By \eqref{310ter} and the Dominated Convergence theorem, we get
\begin{align*}
\iint_{\T^d\times [0,T]}\big(-\partial_t \psi-\Delta\psi+H_p(x,Du;\m)\cdot D\psi\big)\tilde m\, dxdt+\int_{\T^d}\psi m_0dx=0.
\end{align*}
By the uniqueness of the solution to the FP equation, it follows that $\tilde m=m$ and that all the sequence $\{m_n\}$ converges to $m$ and we conclude the continuity of the map $\G$.\\
In conclusion, by Schauder fixed point theorem, we get that there exists a fixed point of the map $\G$.

It remains to prove that such a fixed point fulfills the regularity in~\eqref{eq:regularity} and that $u$ belongs to $C([0,T],C^2(\T^d))$.
The bounds in \eqref{eq:regularity}-(i) are immediate consequences of~$(\cK2)$.
We observe that \eqref{eq:stima_mu} and $(\cK2)$ entail
\begin{equation*}
\Wass(\m(t),\m(s))\le
\frac{\l_1 c_1 c_{\cH}}{1-\l_0}\|D\bu(t)-D\bu(s)\|_{L^\infty}+c_0\frac{1+\l_1(1+\bar K)}{1-\l_0}|t-s|^\half;
\end{equation*}
using $(\cK3)$, we obtain the bound in~\eqref{eq:regularity}-(ii).
Furthermore, Lemma~\ref{lemma:dip_continua} and~\eqref{eq:regularity}-(ii) entail the estimate in~\eqref{eq:regularity}-(iii). Finally, the estimates in~\eqref{eq:regularity}-(iv) are due to~\eqref{eq:bound_1} and~\eqref{eq:bound_2} while estimate~(v) is a straightforward consequence of~\eqref{eq:bound_m}.  Finally, Lemma~\ref{lemma:dip_continua}, $(H3)$ and \eqref{eq:regularity}-(ii) easily entail that $u$ belongs to $C([0,T],C^2(\T^d))$.

\end{proof}

%%%%%%%%%%%%%%%%%% Uniqueness %%%%%%%%%%%%%%%%%%%%%%%%%%%%%%%%%

\begin{theorem}\label{thm:uniqueness}
Under the same hypotheses of Theorem \ref{thm:existence_laxmilgram}, assume also that there exists a constant $\Lambda_0>0$ such that
\begin{align}
&|b(x,a_1;\nu)-b(x,a_2;\nu)|\leq \Lambda_0 |a_1-a_2|\label{hyp:key_assumption1a}
\end{align}
for all $a_1,a_2\in A$, $x\in\T^d$ and $\n\in \cP(\T^d\times A)$. Then, the system \eqref{QSS_discount} admits at most one solution $(u,m,\mu)$ in $C([0,T],C^1(\T^d))\times \cH^1_s(Q)\times C([0,T],\cP(\T^d\times A))$ with $u(t)\in C^2(\T^d)$ for every $t\in[0,T]$.
\end{theorem}
\begin{proof}
Let $(u_1,m_1,\mu_1)$, $(u_2,m_2,\mu_2)$ be two solutions of system~\eqref{QSS_discount} in $C([0,T],C^1(\T^d))\times \cH^1_s(Q)\times C([0,T],\cP(\T^d\times A))$  with $u_i(t)\in C^2(\T^d)$ for every $t\in[0,T]$ and $i\in\{1,2\}$. 
The function $M:=m_1-m_2$ is a solution of
\begin{equation*}
\left\{
\begin{aligned}
&\partial_{t}M - \D M -\diver(M  H_{p}(x,Du_1 ; \m_1 (t))=\\
&\hskip 8pt\diver(m_2( H_{p}(x,Du_1 ; \m_1 (t))- H_{p}(x,Du_2 ; \m_2 (t)))\qquad\mbox{ in } Q \\
& M (0)=0 .
\end{aligned}
\right.
\end{equation*}
In this proof, the letter $C$ will denote a constant that may change from line to line but it is always independent of $(u_1,m_1,\mu_1)$ and $(u_2,m_2,\mu_2)$.\\
Multiplying the previous equation by $M$ and integrating in $\T^d$, we get
\begin{equation}\label{eq:uniq1}
\begin{split}
&\half \frac{d}{dt}\|M \|^2_{L^2(\T^d)}+\int_{\T^d}|DM (t)|^2dx=-\int_{\T^d} \big[M(t) DM(t)\cdot H_{p}(x,Du_1 ; \m_1 )\\
&-m_2(t)DM(t)\cdot (H_p (x,Du_1 ; \m_1 )- H_{p}(x,Du_2 ; \m_2 ))\big]dx.
\end{split}
\end{equation} 
Estimating via Cauchy-Schwarz's inequality the two terms on the left  side of the previous equality, we have
\begin{align*}
&-\int_{\T^d}  M(t) DM(t)\cdot H_{p}(x,Du_1  ; \m_1 )dx\\
&\qquad\le \frac{1}{2 }\int_{\T^d}|M(t)H_{p}(x,Du_1  ; \m_1 )|^2dx +\frac{1}{2}\int_{\T^d}|DM (t)|^2dx	
\end{align*} 
and 
\begin{align*}
&-\int_{\T^d}m_2DM(t)\cdot\Big( H_p(x,Du_1  ; \m_1)- H_{p}(x,Du_2  ; \m_2)\Big)dx\\
&\qquad \le
\frac{1}{2 }\int_{\T^d}|m_2(t)\Big(H_p (x,Du_1  ; \m_1 )- H_{p}(x,Du_2 ; \m_2 )\Big) |^2dx \\&\qquad+\frac{1}{2}
\int_{\T^d}|DM (t)|^2dx.
\end{align*}
By Lemma~\ref{lemma:H1s}, replacing  the previous two inequalities in \eqref{eq:uniq1}, we get
\begin{equation}\label{eq:uniq2}
\begin{split}
&\half \frac{d}{dt}\|M ( t)\|^2_{L^2(\T^d)} \le C\|M ( t)\|^2_{L^2(\T^d)}\\
&+C\|m_2\|^2_{\cH^1_2(Q)}\|H_p (x,Du_1( t) ; \m_1 )-
H_{p}(x,Du_2 (t); \m_2)\|^2_{L^\infty(\T^d)}.
\end{split}
\end{equation}
Recalling   \eqref{eq:optimal_control}, assumptions $(H4)$ and $(H3)$, and  exploiting \eqref{hyp:key_assumption1a}, we have
\begin{align*}
&|H_p(x,Du_1(x,t) ; \m_1)- H_{p}(x,Du_2 (x,t); \m_2 )|\\
&= |b(x,\a^*(x,Du_1(x,t) ; \m_1);\m_1)- b(x,\a^*(x,Du_2 (x,t); \m_2 );\m_2)|\\
&\le \L_0 \big(| \a^*(x,Du_1(x,t) ; \m_1) -\a^*(x,Du_2 (x,t); \m_2 \big)|+  \l_\mu \Wass (\m_1(t),\m_2(t)))\\
&\le C (|Du_1(x,t)-Du_2(x,t) |+\Wass(\m_1 (t),\m_2 (t)))
\end{align*}
and therefore
\begin{multline}\label{eq:uniq3}
\|H_p(x,Du_1(t) ; \m_1)- H_{p}(x,Du_2 (t); \m_2 )\|_{L^\infty}\\
\le C (\|Du_1(t)-Du_2(t)\|_{L^\infty} +   \Wass (\m_1(t),\m_2(t))).
\end{multline}
We now estimate the two terms on the right  side of \eqref{eq:uniq3}.
By Lemma \ref{lemma:fixed_point}.(i),  (ii) and using \eqref{eq:bound_3}, \eqref{eq:bound_m} and Lemma~\ref{lemma:H1s}, we have
\begin{eqnarray}\label{eq:uniq11}
\Wass (\m_1(t),\m_2(t))&\le& \frac{c_1c_\cH\l_1}{1-\l_0}\|Du_1(t)-Du_2(t)\|_{L^\infty} \\ \notag &&+\frac{1+\l_1(1+\bar K)}{1-\l_0}\Wass(m_1(t), m_2(t)).
\end{eqnarray}
Furthermore, Lemma~\ref{lemma:dip_continua} and $(H3)$ entail
\begin{equation}\label{15luglio}
\begin{split}& \|Du_1(t)-Du_2(t)\|_{L^\infty}\leq  (C_0+1) L_\mu\Wass (\mu_1(t),\mu_2(t))\\
&\qquad\leq  (C_0+1) L_\mu\left[\frac{c_1c_\cH\l_1}{1-\l_0}\|Du_1(t)-Du_2(t)\|_{L^\infty} \right.\\
&\qquad\qquad\left. +\frac{1+\l_1(1+\bar K)}{1-\l_0}\Wass(m_1(t), m_2(t))\right]
\end{split}
\end{equation}
where last inequality is due to \eqref{eq:uniq11}. By hypothesis~\eqref{hyp:smallness_Lipschitz}, arguing as in~\eqref{eq:Lstar}, we get
\begin{equation*}
\|Du_1(t)-Du_2(t)\|_{L^\infty}\leq  C\Wass(m_1(t), m_2(t))\quad \forall t\in [0,T];
\end{equation*}
again by estimate~\eqref{eq:uniq11}, we conclude that 
\begin{equation} \label{uniq12}
\|Du_1(t)-Du_2(t)\|_{L^\infty}+\Wass (\m_1(t),\m_2(t))\le C\Wass(m_1(t), m_2(t))\quad \forall t\in [0,T].
\end{equation}
Replacing \eqref{eq:uniq3} and~\eqref{eq:uniq11} in~\eqref{eq:uniq2}, by \eqref{uniq12}, we finally get
\[
\frac{d}{dt}\|M (t)\|^2_{L^2(\T^d)} \le C\|M (t)\|^2_{L^2(\T^d)}.
\]
Since $M(0)=0$, the previous inequality implies that $M(t)=0$ for all $t\in [0, T]$ and therefore $m_1=m_2$.
Inequality \eqref{uniq12} implies that $\m_1=\m_2$ and, by Lemma \ref{lemma:bounds_discount_ergodic}, we also get $u_1=u_2$. 
\end{proof}
%%%%%%%%%%%%%%  Ergodic problem %%%%%%%%%%%%%%%%%%%%%%
We now consider the well-posedness of the ergodic quasi stationary MFG system \eqref{QSS_ergodic}.
\begin{theorem}\label{thm:ergodic1}
Under the same assumptions of Theorem \ref{thm:existence_laxmilgram}, problem  \eqref{QSS_ergodic} admits a    solution $(u,\l,m,\m)$, where $(u,\l)\in C([0,T], C^2(\T^d))\times C([0,T])$ is a classical solution of the  HJB equation for any $t\in [0,T]$, $m\in \cH^1_2(Q)$  is a weak solution of the FP equation  and  $\mu\in C([0,T],\cP(\T^d\times A))$ satisfies the fixed-point equation for any $t\in [0,T]$. 
Moreover, there exists a constant $L$ such that
\begin{equation*} 
\begin{array}{ll}
(i)& \Wass(m(t),m(s))\leq L|t-s|^\half, \quad  \|m\|_{\cH^1_2}\leq L,\\
(ii)& \Wass(\m(t),\m(s))\leq L|t-s|^\half,\\
(iii) & |\l(t)-\l(s)|+\|u(\cdot,t)-u(\cdot,s)\|_{C^2(\T^d)} \leq L|t-s|^\half,\\
(iv) &  \textrm{if moreover $m_0\in W^{1,s}(\T^d)$ with $s>2$, then $\|m\|_{\cH^1_s}\leq L$}.
\end{array}
\end{equation*}
\end{theorem}
\begin{proof}
Let $(u^\r,m^\r,\mu^\r)$ be the solution of \eqref{QSS_discount} corresponding to $\r\in (0,1)$ found in Theorem~\ref{thm:existence_laxmilgram} and set $w^\r=u^\r-  u^\r (0)$. Then $(w^\r,\r u^\r(0), m^\r,\m^\r)$ satisfies the system
\begin{equation*} 
\left\{
\begin{aligned}
&- \D w^\r + H(x,Dw^\r ; \m^\r (t))+\r w^\r +\r u^\r(0)=0 \quad &\mbox{ in } \T^d,\,\forall t\in [0,T]\\
&\partial_{t}m^\r - \D m^\r -\diver(m^\r  H_{p}(x,Dw^\r ; \m^\r (t))=0&\mbox{ in } Q\\
&\m^\r (t)=\left( \Id,\a^{*}(\cdot,Dw^\r(t);\m^\r(t))\right) \sharp m^\r (t)& \forall t\in[0,T] \\
& m^\r (0)=m_{0}& \mbox{ in } \T^d. 
\end{aligned}
\right.
\end{equation*}
We note that an easy application of the comparison principle (see relation~\eqref{eq:dc_1} below for a similar argument) entails
\begin{equation*}
\begin{array}{rl}
\r\|u^\r(\cdot,t)-u^\r(\cdot,s)\|_{L^\infty}\leq & \bar K \max_{x,a}|b(x,a;\m(t))-b(x,a;\m(s))|\\&+\max_{x,a}|\ell(x,a;\m(t))-\ell(x,a;\m(s))|,
\end{array}
\end{equation*}
with $\bar K$ as in~\eqref{eq:bound_2}. By our assumptions on $b$ and $\ell$ and by estimate~\eqref{eq:regularity}-(ii) we infer 
\[
\r\|u^\r(\cdot,t)-u^\r(\cdot,s)\|_{L^\infty}\leq C'|t-s|^{1/2}
\]
for some constant $C'$ independent of~$\r$. Letting $\r\to0^+$, we obtain that $\r u^\r(0,t)\to \l(t)$ with $\|\l\|_{L^\infty}\leq L$ and $|\l(t)-\l(s)|\leq C'|t-s|^{1/2}$.\\
Applying Ascoli-Arzela theorem for the convergence of $m^\r$ and $\m^\r$ (eventually passing to a subsequence that we still denote $(\r u^\r(0),w^\r, m^\r,\m^\r)$), we obtain that as $\r\to0^+$, there hold
\begin{itemize}
\item $\r u^\r(0,t)\to \l(t)$ in the $C([0,T])$-topology,
\item $w^\r$ converges to some $u$ in the $C([0,T],C^2(\T^d))$-topology and
\[
\|u(\cdot,t)-u(\cdot,s)\|_{C^2(\T^d)}\leq C'|t-s|^{1/2},
\]
\item $m^\r$ converge to some $m$ in the $C([0,T],\cP(\T^d))$-topology and
\[
\Wass(m(t),m(s))\leq L|t-s|^\half,
\]
\item $\m^\r$ converge to some $\m$ in the $C([0,T],\cP(\T^d\times A))$-topology and 
\[
\Wass(\m(t),\m(s))\leq L|t-s|^\half.
\]
\end{itemize}
We claim that $(u,\l,m,\m)$ is a solution to problem~\eqref{QSS_ergodic}. Indeed, since the coefficients of~$H$ are continuous w.r.t.~$\m$, as $\r\to0^+$ (with $t$ fixed), standard stability of solutions ensures that $u$ is a solution to the HJB equation in~\eqref{QSS_ergodic}.\\
On the other hand $m^\r$ is a weak solution to the FP equation in~\eqref{QSS_discount}  
 with drift $g^\r(x,t)=b(x,\a^*(x,Du^\r(x,t),\m^\r(t));\m^\r(t))$. Our assumptions on $b$, $(H2)$, \eqref{eq:regularity}-(ii) and (iii) give that $g^\r$ are equibounded and equicontinuous. We can assume (possibly passing to a subsequence) that $g^\r$ uniformly converge to $g(x,t)=b(x,\a^*(x,Du(x,t),\m(t)),\m(t))$. Hence, using the Dominated Convergence theorem in the weak formulation of the solution to the FP equation of~\eqref{QSS_discount} we obtain that $m$ is a weak solution to the FP equation of~\eqref{QSS_ergodic}. 
Finally, it remains to prove that $\m$ solves the fixed point equation in~\eqref{QSS_ergodic}. Indeed, for any $t\in[0,T]$, for any $\phi\in C(\T^d\times A)$, we have
\begin{equation}\label{conv_final}
\int_{\T^d\times A}\phi d \m^\r(t)= \int_{\T^d}\phi^\r (x)m^\r(x,t)\, dx.
\end{equation}
where $\phi^\r (x):=\phi(x,\a^*(x,Du^\r(x,t);\m^\r(t)))$ is a family of bounded, uniformly continuous function which, as $\r\to0^+$, converges uniformly to the function $\phi(x,\a^*(x,Du(x,t);\m(t)))$. Therefore, passing to the limit in~\eqref{conv_final}, we conclude
\begin{equation*}
\int_{\T^d\times A}\phi d\m(t)= \int_{\T^d}\phi(x,\a^*(x,Du(x,t);\m(t)))m(x,t)\, dx
\end{equation*}
which, by the arbitrariness of $\phi$, is equivalent to the fixed point in~\eqref{QSS_ergodic}.
\end{proof}
\begin{theorem} 
Under the same assumptions of Theorem \ref{thm:uniqueness}, the ergodic system  \eqref{QSS_ergodic} admits at most one solution $(u,\l,m,\m)$ in $C([0,T],C^2(\T^d))\times C([0,T])\times \cH^1_s(Q)\times C([0,T],\cP(\T^d\times A))$.
\end{theorem}
\begin{proof}
The proof goes along the same lines of the one of Theorem \ref{thm:uniqueness}; here we just use the second part of the statement of Lemma~\ref{lemma:dip_continua} instead of the first part for obtaining~\eqref{15luglio}.
\end{proof}
%%%%%%%%%%%%%%%%%%%%%%%%%%%%%%%%%%%%%%%%%%%%%%%%%%
%   Well-posedness  via a regular dependence     %
%%%%%%%%%%%%%%%%%%%%%%%%%%%%%%%%%%%%%%%%%%%%%%%%%%

\section{Well-posedness  via a regular dependence  on the the joint distribution of the Hamiltonian}\label{sec:exis_via_reg_dep}
The aim of this section is to obtain existence of a solution to problems~\eqref{QSS_discount} and~\eqref{QSS_ergodic} without requiring a smallness condition for the data as in \eqref{hyp:smallness_Lipschitz} and in $(H4)$ (namely, that $\l_0\in[0,1)$). We assume instead that the Hamiltonian~$H$ satisfies a stability property with respect to the time evolution of the joint distribution~$\{\m(t)\}_{t\in[0,T]}$ (see assumption $(H5)$ below). To this end, we require that~$H$ depends in a nonlocal manner in time on~$\{\m(t)\}_{t\in[0,T]}$. In the next section, we will provide an example where~$H$ satisfies such a stability property, depending only on the past evolution of the joint distribution. This property entails that the behaviour of the agents is not affected by the future evolution of the joint distribution; such a feature seems more realistic for the MFG theory.

We assume that the coefficients of the Hamiltonian $H$ satisfy  
\begin{itemize}
%%%%%  Assumptions for existence %%%%%%%%%%%%%%%%%%%%%%%%%
\item[$(H1')$]
$b$ and $\ell$ are continuous functions on $\T^d\times A\times[0,T]\times C([0,T],\cP(\T^d\times A))$ with value in $\R^d$ and respectively in $\R$. The control set~$A$ is a compact, separable metric space (for simplicity we assume that~$A$ is a subset of some Euclidean space). Moreover, there exist two constants $K$ and $L$ such that
\begin{equation}\label{hyp_b_bis}
\begin{split}
&|b(x,a;t,\nu )|\leq K,\quad |b(x_1,a;t,\nu)-b(x_2,a;t,\nu)|\leq L|x_1-x_2|\\
&|\ell(x,a;t,\nu)|\leq K,\quad |\ell(x_1,a;t,\nu)-\ell(x_2,a;t,\nu)|\leq L|x_1-x_2|
\end{split}
\end{equation}
for all $x,x_1,x_2\in \T^d$, $a\in A$, $\nu\in C([0,T],\cP(\T^d\times A))$.
\end{itemize}
Now, the Hamiltonian~$H$ in \eqref{eq:hamiltonian} is a function on $\T^d\times \R^d\times[0,T]\times C([0,T],\cP(\T^d\times A))$ with value in $\R$ defined as
\begin{equation}\label{4agosto}
H(x,p;t,\n):=\sup_{a\in A}\left\{ -p\cdot b(x,a;t,\nu)-\ell(x,a;t,\nu)\right\}.
\end{equation}
\begin{itemize}
\item[$(H2')$]  For any $(x,p,t,\n)\in\T^d\times\R^d\times[0,T]\times C([0,T],\cP(\T^d\times A))$, there exists a unique $\a^*=\a^*(x,p;t,\nu)\in A$ such that 
\[H(x,p;t,\nu)=  -p\cdot b(x,\a^*;t,\nu)-\ell(x,\a^*;t,\nu). \]	
Moreover the map $\a^*$ is continuous with respect to its arguments.

\item[$(H3')$] For each positive constant $R$, there exists a constant~$L_{\mu,R}$ such that
\begin{multline*}
\max_{x,a}|(b(x,a;t,\nu)-b(x,a;t',\nu'))p| +\max_{x,a}|\ell(x,a;t,\nu)-\ell(x,a;t',\nu')| \\
\le  L_{\mu,R}|t-t'|+L_{\mu,R}\,\sup_{r\in [0,t]}\Wass(\n(r),\n'(r))
\end{multline*}
for any $p\in B(0,R)$, $t,t'\in[0,T]$ and $\n,\n'\in C([0,T],\cP(\T^d\times A))$.\\
As in $(H3)$, we denote $L_\mu$ the constant $L_{\mu,\bar K}$  where $\bar K$ is the constant introduced in Lemma~\ref{lemma:bounds_discount_ergodic}.
\item[$(H4')$]
There exist $\l_0,\l_1\in [0,\infty)$   such that
\begin{multline*}
|\a^*(x,p;t,\nu)- \a^*(x',p';t',\nu')|\le \l_0\sup_{r\in [0,t]}\Wass(\n(r),\n'(r))\\+\l_1(|x-x'|+|p-p'|+|t-t'|)
\end{multline*}
for all $x,x'\in \T^d$, $p,p'\in \R^d$, $t,t'\in [0,T]$, $\n,\n'\in C([0,T],\cP(\T^d\times A))$.
\end{itemize}
\begin{remark}
Note that $(H4')$ does not require $\l_0\in[0,1)$ as in $(H4)$.
\end{remark}
We introduce the following assumption which concerns the  stability of $H$ w.r.t. the joint distribution $\m$
\begin{itemize}
\item[(H5)] For any $R>0$   and for every sequence $\m_n\in C([0,T],{\cal P}(\T^d\times A))$, the family of functions $(x,p,t)\mapsto H(x,p;t, \m_n)$ is sequentially compact in the topology of the  uniform convergence on $\T^d\times \overline{B(0,R)}\times [0,T]$.
\end{itemize}
In Example 2 of Section \ref{sec:comments_examples}, we provide an Hamiltonian which satisfies the previous assumptions and only depends on the past evolution of the joint distribution state-control (memory effect).\\
We establish a time-regularity result (independent of the measure $\m$) for the solution $u$ of the HJB equation in \eqref{QSS_discount}. Recall that this equation  is stationary and depends on the time only through the measure $\m$. For the proof of the next result, see the Appendix.
\begin{lemma}\label{lemma:new}
Assume $(H1')$, $(H2')$ and $(H5)$.  For $\m\in C([0,T],\cP(\T^d\times A))$ and $\r\in(0,1)$, let $u$ be the solution to
\begin{equation}\label{HJB_discount}
- \D u + H(x,Du ; t, \m)+\r u =0 \quad \mbox{ in } \T^d,\,\forall t\in [0,T].
\end{equation}
Then, there exists a modulus of continuity $\o$, independent of~$\r\in(0,1)$, such that for any $\m\in C([0,T],\cP(\T^d\times A))$, there holds
\begin{equation}\label{est_grad}	
\|Du(t)-Du(s)\|_{L^\infty}\leq \o(|t-s|)\qquad\forall t,s\in[0,T].
\end{equation}
\end{lemma}
The results of  Lemmas \ref{lemma:bounds_discount_ergodic} and \ref{lemma:dip_continua} still hold with the constants uniform in $t\in [0,T]$ and $\nu\in C([0,T],{\cal P}(\T^d\times A))$ (exploiting assumption~$(H5)$).

%%%%%%%%%%%%%%%%%%%%%%%%%%%%%%%
%  without monotonicity       %
%%%%%%%%%%%%%%%%%%%%%%%%%%%%%%%
In Theorem \ref{thm:existence_laxmilgram},  to prove existence of a solution, we used two fixed point maps, one for the equation \eqref{eq:fixed_1}  and one to get the existence  of a  solution to the MFG system. Here we shall follow  a completely different approach, which relies on a unique fixed point argument.
\begin{theorem}\label{thm:1pt_fisso}
Assume $(H1')$, $(H2')$, $(H3')$, $(H4')$, $(H5)$ and $m_0\in H^{1}(\T^d)$. Then, problem \eqref{QSS_discount} admits a solution~$(u,m,\m)$ where $u\in C([0,T], C^2(\T^d))$ is a classical solution of the  HJB equation for any $t\in [0,T]$, $m\in \cH^1_2(Q)$  is a weak solution of the FP equation  and  $\mu\in C([0,T],\cP(\T^d\times A))$ satisfies the fixed point equation for any $t\in [0,T]$.\\
Moreover, there exists a positive constant $L$ and a modulus of continuity~$\o$ (both independent of $\rho$) such that, for any $t,s\in[0,T]$, there hold:
\begin{equation}\label{eq:regularity2}
\begin{array}{ll}
(i)& \Wass(m(t),m(s))\leq L|t-s|^\half, \quad  \|m\|_{\cH^1_2}\leq L,\\
(ii)&  \Wass(\m(t),\m(s))\leq \o(|t-s|),\\
(iii)& \|[u(\cdot,t)-u(0,t)]-[u(\cdot,s)-u(0,s)]\|_{C^2(\T^d)} \leq \o(|t-s|),\\
(iv) & \|\r u\|_{L^\infty}\leq L,\quad \|u(\cdot,t)-u(0,t)\|_{C^2(\T^d)}\leq L,\\
(v) & \textrm{if moreover $m_0\in W^{1,s}(\T^d)$ with $s>2$, then $\|m\|_{\cH^1_s}\leq L$}.
\end{array}
\end{equation}
\end{theorem}
\begin{proof}
We consider the  map $\Psi$ that, to each $\m\in C([0,T], \cP(\T^d\times A))$, associates
\begin{equation*} 
\Psi(\mu)(t):= \left( \Id,\a^*(\cdot,Du(t);t, \m)\right) \sharp m(t),\qquad t\in [0,T],
\end{equation*}
where the functions $u$ and $m$ are the solutions of HJB equation and of the FP equation in~\eqref{QSS_discount} corresponding to $\m$.
We introduce the set
\begin{equation*} 	
\cK=\left\{
\begin{array}{l}
\m\in C([0,T], \cP(\T^d\times A)):\text{ for any $s,t\in[0,T]$,} \\[4pt]
(\cK1)\, \int_{\T^d\times A}\left[|x|^2+|a|^2\right]d\m(t)(x,a)<c_2\\[4pt]
(\cK2)\, \Wass(\m(t),\m(s))\le \o_\cK(|t-s|)
\end{array}
\right\}
\end{equation*}
 endowed with the topology of $C([0,T], \cP(\T^d\times A))$, where $c_2$ is a constant and $\o_\cK$ a modulus of continuity, both independent of $\r\in(0,1]$, that will be suitably chosen later on.\\
By Ascoli-Arzela theorem, the set~$\cK$ is compact. Moreover, it is convex and not empty.
Let us assume for the moment that
\begin{itemize}
\item[(a)] $\Psi$ maps $\cK$ into itself;
\item[(b)] $\Psi$ is a continuous map from $C([0,T], \cP(\T^d\times A))$ into itself. 
\end{itemize}
Invoking Schauder fixed point theorem for the map $\Psi$ on the set $\cK$, we obtain a solution to problem~\eqref{QSS_discount} with $\m\in\cK$.\\
Let us now establish the bounds in~\eqref{eq:regularity2}. Point~$(i)$ is due to Lemma~\ref{lemma:linear_FP}. Point~$(ii)$ is equivalent to~$(\cK2)$. Point~$(iii)$ is an easy consequence of point~$(ii)$, assumption~$(H3')$ and Lemma~\ref{lemma:dip_continua}. Point~$(iv)$ is due to Lemma~\ref{lemma:bounds_discount_ergodic}. Point~$(v)$ follows from~\eqref{eq:bound_m}.

It remains to prove properties $(a)$ and $(b)$.\\
$(a)$.
Let us first prove that
\begin{equation}\label{3giu_7}
\Psi(\mu)\in C([0,T], \cP(\T^d\times A)).
\end{equation}
Indeed, in order to prove this property, it is enough to prove $\Psi(\m)(s)\to \Psi(\m)(t)$ in the  weak$^*$-topology as $s\to t$, namely 
\begin{equation}\label{3giu_pag2}
\lim_{s\to t}\int_{\T^d\times A}\phi(x,\a)\Psi(\m)(s)(dx,d\a)=\int_{\T^d\times A}\phi(x,\a)\Psi(\m)(t)(dx,d\a)
\end{equation}
for any bounded and continuous function~$\phi=\phi(x,\a)$.
Actually, we have
\[
\int_{\T^d\times A}\phi(x,\a)\Psi(\m)(s)(dx,d\a)=
\int_{\T^d}\phi(x,\a^*(x,Du(x,s);s, \m))m(s)dx
\]
and similarly for the right hand side of~\eqref{3giu_pag2}. We observe that Lemma~\ref{lemma:dip_continua} and $(H3')$ ensure
\[
\lim_{s\to t}\|Du(s)-Du(t)\|_{L^\infty}=0.
\]
Taking   into account  $(H4')$, we have
\[
\lim_{s\to t}\a^*(x,Du(x,s);s,\m)= \a^*(x,Du(x,t);t,\m)\qquad \textrm{uniformly in $x$.}
\]
By~\eqref{eq:holder_Wass}, we also have $\Wass(m(s),m(t))\to 0$ and we deduce relation~\eqref{3giu_pag2}; hence, property~\eqref{3giu_7} is completely proved.\\
%%%%%%%%%%%%%%%%%%%%%%%%%%%%%%%%%%%%%%%%%
Since $\T^d\times A$ is compact, for $c_2$ sufficiently large (and independent of~$\r$), it is obvious that~$\Psi(\m)$ fulfills~$(\cK1)$.
In order to prove $(\cK2)$, we first claim
\begin{multline}\label{3giu_8}
\Wass(\Psi(\m)(s),\Psi(\m)(t))\leq 
c_0(1+\l_1+\bar K\l_1)|t-s|^{1/2}+\l_1(\o(|t-s|)+|t-s|),
\end{multline}
where $\o$ is the modulus of continuity found in Lemma~\ref{lemma:new}.
Indeed, by definition of $\Wass$ and of $\Psi(\mu)$, (recall that, by Lemma \ref{lemma:linear_FP}, for every $t\in[0,T]$, the measure $m(t)$ has a density), we have
\begin{equation}\label{3giu_pag3}
\begin{array}{l}
\Wass(\Psi(\m)(s),\Psi(\m)(t))\\
=\sup \int_{\T^d} [\phi(x,\a^*(x,Du(x,s);s,\m))m(s)-\phi(x,\a^*(x,Du(x,t);t,\m))m(t)] dx\\
\leq \sup \int_{\T^d}\phi(x,\a^*(x,Du(x,s);s,\m))(m(s)-m(t))\, dx\\
\quad+\sup \int_{\T^d}\left[\phi(x,\a^*(x,Du(x,s);s,\m))-\phi(x,\a^*(x,Du(x,t);t,\m))\right]m(t)\, dx\\
=: \sup A_1 +\sup B_1
\end{array}
\end{equation}
where the suprema are performed over all $1$-Lipschitz function $\phi$.
We observe that, by assumption $(H4')$ and \eqref{eq:bound_2}, the map
\[
x\mapsto \tilde\phi(x):=\phi(x,\a^*(x,Du(x,s);s,\m)) 
\]
is $(1+\l_1+\bar K\l_1)$-Lipschitz continuous; in particular, by~\eqref{eq:holder_Wass}, we infer 
\begin{equation}\label{3giu_pag3A}
\begin{array}{rcl}
\sup A_1
&\leq& (1+\l_1+\bar K\l_1)\Wass(m(s),m(t))\\
&\leq& c_0(1+\l_1+\bar K\l_1)|t-s|^{1/2}.
\end{array}
\end{equation}
On the other hand, we have
\begin{equation}\label{3giu_pag3B}
\begin{array}{rcl}
\sup B_1 &\leq &\|\phi(\cdot,\a^*(\cdot,Du(s);s,\m))-\phi(\cdot,\a^*(\cdot,Du(t);t,\m))\|_{L^\infty}\\
&\leq & \sup_{x\in\T^d}|\a^*(x,Du(x,s);s,\m)-\a^*(x,Du(x,t);t,\m)|\\
&\leq & \l_1\|Du(s)-Du(t)\|_{L^\infty}+\l_1|t-s|\\
&\leq & \l_1 (\o(|t-s|)+|t-s|)
\end{array}
\end{equation}
where the last two inequalities are due to assumption $(H4')$ and respectively Lemma~\ref{lemma:new}.
Replacing inequalities~\eqref{3giu_pag3A} and~\eqref{3giu_pag3B} in~\eqref{3giu_pag3}, we accomplish the proof of estimate~\eqref{3giu_8}.\\
If we choose $\o_\cK$ in $(\cK2)$ as
\[\o_\cK(r)=c_0(1+\l_1+\bar K\l_1)|r|^{1/2}+\l_1(\o(|r|)+|r|),\]
then, by \eqref{3giu_8}, we get that $\Psi(\mu)$ satisfies also $(\cK2)$. 

\noindent $(b)$. For $\m\in C([0,T],\cP(\T^d\times A))$, let $u$ and $m$ be the corresponding solutions of the HJB equation and respectively of the FP equation in~\eqref{QSS_discount}. Consider a sequence $\m_n\in C([0,T], \cP(\T^d\times A))$ such that, as $n\to+\infty$, $\m_n\to \m$ in the $C([0,T], \cP(\T^d\times A))$ topology. We want to prove that $\Psi(\m_n)\to \Psi(\m)$ in the same topology.
We denote $u_n$ and $m_n$ respectively the solutions to the first two equations in~\eqref{QSS_discount} with $\m$ replaced by $\m_n$; since now on, $o_n(1)$ stands for a function on $[0,T]$ (which may change from line to line) such that $\lim_{n\to\infty} o_n(1)=0$ uniformly in $[0,T]$.\\
We have $\sup_{[0,T]}\Wass(\m_n(t),\m(t))=o_n(1)$; hence, by Lemma~\ref{lemma:dip_continua} and $(H3')$, we deduce
\begin{equation*}
\|Du_n(t)-Du(t)\|_{L^\infty}=o_n(1).
\end{equation*}
Moreover, by \eqref{eq:optimal_control}, $(H1')$, $(H4')$ and boundedness of $Du_n$ in \eqref{eq:bound_2}, we get
\begin{equation*}
  \|H_p(\cdot,Du_n(t);t, \m_n)-H_p(\cdot,Du(t);t, \m)\|_{L^\infty}=o_n(1).
\end{equation*}
By estimates~\eqref{eq:holder_Wass} and \eqref{eq:moment_m}, (possibly passing to a subsequence) as $n\to\infty$ the sequence $\{m_n\}$ converges to some function~$\tilde m$ in the $C([0,T], \cP(\T^d\times A))$-topology. By stability, the function~$\tilde m$ solves the Fokker-Planck equation in~\eqref{QSS_discount} which admits a unique solution; therefore, $\tilde m$ coincides with~$m$ and, as $n\to\infty$, the whole sequence~$\{m_n\}$ converges to~$m$ in the $C([0,T], \cP(\T^d\times A))$-topology.
Moreover, there holds
\begin{equation}\label{3giu_pag5}
\begin{array}{l}
\Wass(\Psi(\m_n)(t),\Psi(\m)(t))=\sup \int_{\T^d\times A}\phi(x,\a)[\Psi(\m_n)(t)-\Psi(\m)(t)](dx,d\a)\\
\quad=\sup\int_{\T^d}\left[\phi(x,\a^*(x,Du_n(x,t);t, \m_n))m_n(t,x)\right.\\
\qquad\left.-\phi(x,\a^*(x,Du(x,t);t, \m))m(t,x)\right]\, dx\\
\quad\leq \sup\int_{\T^d}\phi(x,\a^*(x,Du_n(x,t);t, \m_n))[m_n(t,x)-m(t,x)]\, dx\\
\qquad +\sup\int_{\T^d}\big[\phi(x,\a^*(x,Du_n(x,t);t, \m_n))\\
\qquad-\phi(x,\a^*(x,Du(x,t);t, \m))\big]m(t,x)\, dx
 =: \sup A_2 +\sup B_2,
\end{array}
\end{equation}
where the suprema are performed over all $1$-Lipschitz function $\phi$.
We observe that, by assumption $(H4')$ and \eqref{eq:bound_2}, the map
\[
x\mapsto \phi_n(x):=\phi(x,\a^*(x,Du_n(x,t);t,\m_n)) 
\]
is $(1+\l_1+\l_1\bar K)$-Lipschitz continuous; in particular, we infer
\begin{equation*}
\begin{array}{rcl}
\sup A_2
&\leq& (1+\l_1+\l_1\bar K)\Wass(m_n(t),m(t))=o_n(1).
\end{array}
\end{equation*}
On the other hand, by assumption $(H4')$, we have
\begin{equation*}
\begin{array}{rcl}
\sup B_2&\leq &\|\phi(\cdot,\a^*(\cdot,Du_n(t);t, \m_n))-\phi(\cdot,\a^*(\cdot,Du(t);t, \m))\|_{L^\infty}\\
&\leq & \|\a^*(\cdot,Du_n(t);t, \m_n)-\a^*(\cdot,Du(t);t, \m)\|_{L^\infty}\\
&\leq & \l_1\|Du_n(t)-Du(t)\|_{L^\infty}+\l_0\sup_{[0,t]}\Wass(\m_n(r),\m(r))=o_n(1).
\end{array}
\end{equation*}
Replacing the last two inequalities in \eqref{3giu_pag5}, we get $\Wass(\Psi(\m_n)(t),\Psi(\m)(t))=o_n(1)$ namely, $\Psi(\m_n)\to \Psi(\m)$ in the $C([0,T], \cP(\T^d\times A))$ topology. The proof of point~$(b)$ is achieved.
\end{proof}
\begin{remark}
An uniqueness result similar to that of Theorem \ref{thm:uniqueness} under the hypotheses of this section, in particular (H5), seems more difficult to be obtained. Indeed, in Section~\ref{sec:exist_via_Cont_dep}, the smallness condition played a crucial role in order to obtain a continuous dependence of the joint distribution~$\m$ w.r.t. the data of the third equation as in Lemma~\ref{lemma:fixed_point}. Here, without this property, we cannot have the well-posedness (existence, uniqueness and continuous dependence) for the third equation; hence, we cannot %the time H\"older continuity of the joint distribution and, in turns, of the solution to the HJB equation. Here, without this condition, we only obtain a uniform continuity in time of these two components and this regularity does not allow to
apply an argument based on the Gronwall's inequality. At the moment, we are unable to find an alternative proof.
\end{remark}
We also establish an existence result for the ergodic system \eqref{QSS_ergodic}. The proof follows the same arguments of the one of Theorem~\ref{thm:ergodic1} and relies on estimates~\eqref{eq:regularity2} so we shall omit it.
\begin{theorem}\label{thm:ergodic_H5}
Under the same assumptions of Theorem \ref{thm:1pt_fisso}, problem  \eqref{QSS_ergodic} admits a solution $(u,\l,m,\m)$ in $C([0,T],C^2(\T^d))\times C([0,T])\times\cH^1_2(Q)\times C([0,T], \cP(\T^d\times A))$. Moreover, there exists a constant $L$ such that
\begin{equation*} 
\begin{array}{ll}
(i)& \Wass(m(t),m(s))\leq L|t-s|^\half, \quad  \|m\|_{\cH^1_2}\leq L,\\
(ii)& \Wass(\m(t),\m(s))\leq \omega(t-s),\\
(iii) & |\l(t)-\l(s)|+\|u(\cdot,t)-u(\cdot,s)\|_{C^2(\T^d)} \leq \omega(t-s),\\
(iv) &  \textrm{if moreover $m_0\in W^{1,s}(\T^d)$ with $s>2$, then $\|m\|_{\cH^1_s}\leq L$}.
\end{array}		
\end{equation*}
\end{theorem}
%%%%%%%%%%%%%%%%%%%%%%%%%%%
% with  monotonicity      %
%%%%%%%%%%%%%%%%%%%%%%%%%%%

\section{Examples}\label{sec:comments_examples}
This section contains some examples where the assumptions of previous sections are satisfied. Moreover, we study the particular case where the Hamiltonian depends separately on the joint distribution.

 %%%% Esempio regolarità controllo ottimo %%%%%%%%%%%%%%
\textbf{Example 1: }
We describe an example of Hamiltonian satisfying $(H4)$. For $A=\{a\in\R^d:\,|a|\le R\}$, let 
\begin{equation*}
b(x,a;\nu)=b_0(x,\nu)-a\qquad\textrm{and}\qquad
\ell(x,a;\nu)=\frac{|a|^2}{2\ell_0(x;\nu)}.
\end{equation*} 
We have
\[
\a^*(x,p;\nu)=
\begin{cases}
\ell_0(x;\nu)p\quad &\textrm{for}\quad|p|\le \dfrac{R}{\ell_0(x;\nu)}\\[10pt]
R\dfrac{p}{|p|}&\textrm{for}\quad|p|>\dfrac{R}{\ell_0(x;\nu)},
\end{cases}
\]
and
\[
H(x,p;\nu)=
\begin{cases}
\ell_0(x;\nu)\dfrac{|p|^2}{2}-b_0(x;\nu)p\quad &\textrm{for}\quad|p|\le \dfrac{R}{\ell_0(x;\nu)}\\[6pt]
R|p|-\dfrac{R^2}{2\ell_0(x;\nu)}-b_0(x;\nu)p&\textrm{for}\quad|p|>\dfrac{R}{\ell_0(x;\nu)}.
\end{cases}
\]
Assume that $\ell_0$, $b_0$ are Lipschitz continuous, bounded and $\ell_0(x;\nu)\ge \delta>0$.  For $|p|\le \dfrac{R}{\ell_0(x;\nu)}$, we have
\[   
|\a^*(x,p;\n_1)-\a^*(x,p;\n_2)|\le \frac{R}{\d}|\ell(x_0;\n_1)-\ell(x_0;\n_2)|\le \frac{RL_0}{\d}\Wass(\n_1,\n_2)
\]
where $L_0$ is the Lipschitz constant of $\ell_0$ with respect to $\nu$. Hence, the condition  \eqref{hyp:key_assumption2} is satisfied for $\l_0:=RL_0/\d<1$.

\medskip
%%%%%%%%%%%%%%%% Example memory %%%%%%%%%%%%%%%%%%%%
\textbf{Example 2: }
We now provide an example of Hamiltonian as in~\eqref{4agosto} which fulfills assumptions~$(H3')$ and~$(H5)$ and whose coefficients only depend on the past evolution of the joint distribution. Here, ${\mathcal M}(\T^d \times A)$ stands for the set of non negative Borel measures on $\T^d\times A$ endowed with the distance
\begin{equation*}
\Wassgen(\m,\n)=\sup_\phi \int_{\T^d\times A}\left(\phi(x,\a)-\phi(0,0)\right)\left(\m-\n\right)(dx,d\a)
\end{equation*}
where the supremum is performed on the set of $1$-Lipschitz continuous functions. Note that, for $\m,\n\in\cP(\T^d\times A)$, $\Wassgen(\m,\n)$ coincides with~$\Wass(\m,\n)$.
Fix a $\R$-valued non negative function~$K\in C([0,T])$. For $\n\in C([0,T];\cP(\T^d\times A))$, we set
\begin{equation*}
[\n](s)=\int_0^sK(\tau)\n(\tau)\, d\tau \qquad \forall s\in[0,T];
\end{equation*}
in other words: $[\n](s)\in {\mathcal M}(\T^d \times A)$ with $[\n](s)(Y)=\int_0^sK(\tau)\n(\tau)(Y)\, d\tau$ for any Borel set $Y\subset \T^d\times A$.\\
Let $\tilde b$ and $\tilde \ell$ be two functions defined on $\T^d\times A\times \cM(\T^d\times A)$ with values in~$\R^d$ and respectively in~$\R$, which satisfy $(H1)$ with $\nu \in {\mathcal M}(\T^d \times A)$ and
\begin{equation*}
|\tilde b(x,a;\m)-\tilde b(x,a;\n)|+|\tilde \ell(x,a;\m)-\tilde \ell(x,a;\n)|\leq L \Wassgen(\m,\n) 
\end{equation*}
for any $x\in\T^d$, $a\in A$, $\m,\n\in \cM(\T^d \times A)$;
we introduce
\begin{equation*}
b(x,a;t,\n):=\tilde b(x,a;[\n](t)),\qquad \ell(x,a;t,\n):=\tilde \ell(x,a;[\n](t)).
\end{equation*}
Let us show that assumption~$(H5)$ is fulfilled. To this end, consider a sequence $\{\m_n\}_n$, with $\m_n\in C([0,T];\cP(\T^d\times A))$; we want to prove that the sequence $\{H_n\}_n$, with $H_n(x,p,t):=H(x,p;t,\m_n)$, is sequentially compact on $\T^d\times \overline{B(0,R)}\times[0,T]$. To this end we observe that, by~$(H1)$, the $H_n$'s are uniformly bounded on $\T^d\times \overline{B(0,R)}\times[0,T]$. On the other hand, again by $(H1)$, we have
\[
|H_n(x_1,p_1,t)-H_n(x_2,p_2,t)|\leq K|p_1-p_2|+L(R+1)|x_1-x_2|
\]
for every $(x_1,p_1,t), (x_2,p_2,t)\in \T^d\times \overline{B(0,R)}\times[0,T]$ and $n\in \N$. Moreover, we have
\[
|H_n(x,p,t_1)-H_n(x,p,t_2)|\leq C\Wassgen([\m_n](t_1),[\m_n](t_2))
\]
where $C$ is a suitable constant. We observe
\begin{align*}
&\Wassgen([\m_n](t_1),[\m_n](t_2))\\
&\qquad  =\sup_\phi\int_{t_2}^{t_1}K(\tau) \left\{\int_{\T^d\times A}\left(\phi(x,\a)-\phi(0,0)\right)\m_n(\tau)(dx,d\a)\right\}\, d\tau\\
&\qquad \leq \|K\|_\infty \|\phi(\cdot,\cdot)-\phi(0,0)\|_{L^\infty}|{t_2}-{t_1}|\leq C' |{t_2}-{t_1}|
\end{align*}
for a constant~$C'$ depending only on~$K$, $d$ and the diameter of~$A$.
By the last three inequalities, the functions $H_n$ are also uniformly continuous on $\T^d\times \overline{B(0,R)}\times[0,T]$. By Ascoli-Arzela theorem we conclude that the sequence~$\{H_n\}$ is sequentially compact; hence, assumption~$(H5)$ is fulfilled.\\
Assumption~$(H3')$ is also verified; the proof is similar to the above arguments so we shall omit it.

\paragraph{Separated dependence on the joint distribution.} 
We assume $(H1)$, $(H2)$, $(H3)$ and 
\begin{equation}\label{separate_cost}
b=b(x,a)\qquad\textrm{and}\qquad\ell(x,a;\m)=\ell_0(x,a)-\ell_1(\m).
\end{equation}
Now the Hamiltonian reads
\begin{equation*}
H(x,p;\mu)=\sup_{a\in A}\left\{ -p\cdot b(x,a)-\ell_0(x,a)\right\}+\ell_1(\m)=:H_0(x,p)+\ell_1(\m).
\end{equation*}
In this setting, systems~\eqref{QSS_discount} and~\eqref{QSS_ergodic} almost decouple; indeed, the HJB equation is independent of the other two equations (more precisely, $Du(t)$ is independent of $\m(t)$) and one has only to solve the systems of the last two equations. This property permits to establish the existence of a solution of the two systems~\eqref{QSS_discount} and~\eqref{QSS_ergodic} without requiring neither smallness of the constants (as in Section~\ref{sec:exist_via_Cont_dep}) nor the stability property~$(H5)$ (as in Section~\ref{sec:exis_via_reg_dep}). To this end, we start with a simple, but useful, observation on the HJB equation.
\begin{remark}\label{prp:5.1}
For $\n_1,\n_2\in \cP(\T^d\times A)$, let $u^\r_i$ solve problem~\eqref{eq:HJB_bound_discount} with $\n$ replaced by~$\n_i$, for $i=1,2$. Then,
\[
u^\r_1(x)=u^\r_2(x) +(\ell_1(\n_1)-\ell_1(\n_2))/\r \qquad \forall x\in\T^d. 
\]
Moreover, let $u_i$ solve~\eqref{eq:HJB_bound_ergodic} with $\n$ replaced by~$\n_i$, for $i=1,2$. Then, $u_1=u_2$.\\
Actually, uniqueness result for~\eqref{eq:HJB_bound_discount} yields $u^\r_1(\cdot)=u^\r_2(\cdot) +(\ell_1(\n_1)-\ell_1(\n_2))/\r$. In particular, $u^\r_1(\cdot)-u^\r_1(0)=u^\r_2(\cdot)-u^\r_1(0)$. Letting $\r\to0$, we get $u_1=u_2$.
\end{remark}
\begin{proposition}\label{thm:separated}
The results of Theorem~\ref{thm:1pt_fisso} and of Theorem~\ref{thm:ergodic_H5} hold true.
\end{proposition}
\begin{proof}
By Remark~\ref{prp:5.1}, the results in Lemma~\ref{lemma:new} are verified; more precisely, the estimate~\eqref{est_grad} is fulfilled with~$\o=0$. The rest of the proof follows, with some easy adaptations, the ones of Theorem~\ref{thm:1pt_fisso} and of Theorem~\ref{thm:ergodic_H5}.
\end{proof}

%%%%%%%%%%%%%%%%%%%%%%%%%%%%%%%
%       Appendix              %
%%%%%%%%%%%%%%%%%%%%%%%%%%%%%%%
\appendix
\section{Appendix: proofs of some results}\label{sec:appendix}
\subsection{Proof of Lemma \ref{lemma:dip_continua}}
\begin{proof}
Set $b^i(x,a)=b(x,a;\n_i)$ and $\ell^i(x,a)=\ell(x,a;\n_i)$, $i=1,2$. We first claim that
\begin{equation}\label{eq:dc_1}
\r\|u^\r_1-u^\r_2\|_{L^\infty}\le  C  \max_{x,a}|b^1(x,a)-b^2(x,a)|+\max_{x,a}|\ell^1(x,a)-\ell^2(x,a)|,
\end{equation}
where $C=C_1(1+K+L)$ as in \eqref{eq:bound_2}. Indeed, to prove  the claim, it is sufficient to observe that 
\[v^\r_{\pm}(x)= u^\r_2(x)\pm \r^{-1}\left( C\max_{x,a}|b^1(x,a)-b^2(x,a)|+\max_{x,a}|\ell^1(x,a)-\ell^2(x,a)|\right) \]
are a subsolution and a supersolution of the equation satisfied by $u^\r_1$. Then estimate~\eqref{eq:dc_1} follows by the  comparison principle. The rest of the proof follows the corresponding argument in \cite[Thm. 2.2]{marchi}.
\end{proof}

%%%%%%%%%%%%%%%%%%%%%%%%%%%%%%%%%%
\subsection{Proof of Lemma \ref{lemma:fixed_point}}
\begin{proof}
Given $m$ and $p$ as in the statement, let $\Psi:\cP(\T^d\times A)\to\cP(\T^d\times A)$ be the map defined by $\Psi(\m)=\left( \Id,\a^*(\cdot,p(\cdot);\m)\right) \sharp m$. Given $\m_1,\m_2 \in \cP(\T^d\times A)$, by  \eqref{hyp:key_assumption2}, we have
\begin{align*}
&\Wass(\Psi(\m_1),\Psi(\m_2))= \sup_\phi\left\{\int_{\T^d\times A}	\phi(x,a)d(\Psi(\m_1) -\Psi(\m_2))\right\}\\
&=\sup_\phi\left\{ \int_{\T^d}\Big(\phi(x,\a^*(x,p(x) ; \m_1 ))- \phi(x,\a^*(x,p(x); \m_2))\Big)m(x)dx\right\}\\
&\le\|\a^*(\cdot,p(\cdot) ; \m_1 ) -\a^*(\cdot,p (\cdot); \m_2 )\|_{L^\infty}\le  \l_0\Wass(\m_1,\m_2),
\end{align*}
where the $\sup$ here  and in the following formulas is taken with respect to $1$-Lipschitz functions on $\T^d\times A$.
Hence, by $(H4)$, $\Psi$ is a contraction and therefore there exists a unique fixed point to \eqref{eq:fixed_1}.\\
$(i)$. By~\eqref{hyp:key_assumption2}, we have
\begin{align*}
\Wass(\m_n,\m)&=\sup_\phi\left\{ \int_{\T^d}\left[\phi(x,\a^*(x,p_n(x) ; \m_n))- \phi(x,\a^*(x,p(x); \m))\right]m(x)dx\right\}\\
&\le\int_{\T^d}\l_1|p_n(x)-p(x)|m(x)dx+\l_0\Wass(\m_n,\m).
\end{align*} 
We deduce
\begin{equation}\label{eq:estimate_wass_m}
\Wass(\m_n,\m)\leq (1-\l_0)^{-1}\l_1\int_{\T^d}|p_n(x)-p(x)|m(x)dx
\end{equation}
and the statement is an easy consequence of  the Cauchy-Schwarz inequality applied to~\eqref{eq:estimate_wass_m}.\\
$(ii)$. Consider $m_n$ and $m$, $\m_n$ and $\m$ as in the statement. We have
\begin{align*}
&\Wass(\m_n,\m)= \sup_\phi\left\{\int_{\T^d\times A}\phi(x,a)d(\m_n-\m)\right\}\\&
=\sup_\phi\left\{ \int_{\T^d}\left[\phi(x,\a^*(x,p(x) ; \m_n))m_n(x)- \phi(x,\a^*(x,p(x); \m))m(x)\right]dx\right\}\\
&\le \sup_\phi\left\{ \int_{\T^d}\left[\phi(x,\a^*(x,p(x) ; \m_n))- \phi(x,\a^*(x,p(x); \m))\right]m_n(x) dx\right\}
\\&\qquad +\sup_\phi\left\{ \int_{\T^d}\phi(x,\a^*(x,p(x); \m))(m_n(x)-m(x))dx\right\}.
\end{align*}
We denote by $I_1$ and $I_2$ the two terms in the last side. Assumption~\eqref{hyp:key_assumption2} ensures
\[
I_1\leq \l_0\Wass(\m_n,\m).
\]
On the other hand, by~\eqref{hyp:key_assumption2}, the function $\a^*(\cdot,p(\cdot);\m)$ is Lipschitz continuous with Lipschitz constant $1+\l_1(1+ L_p)$. We deduce
\begin{eqnarray*}
I_2&=&(1+\l_1+\l_1L_p)\sup_\phi\left\{ \int_{\T^d}\frac{\phi(x,\a^*(x,p(x); \m))}{1+(\l_1+\l_1L_p)}(m_n(x)-m(x))dx\right\}\\
&\leq&  [1+\l_1(1+L_p)]\Wass (m_n,m),
\end{eqnarray*}
where  the last relation is due to the fact that the integrand is a $1$-Lipschitz continuous function in~$x$. Replacing the last two relations in the previous one we obtain the statement.
\end{proof}

%%%%%%%%%%%%%%%%%%%%%%
\subsection{Proof of Lemma \ref{lemma:new}}
\begin{proof}
We shall borrow some arguments of~\cite[Lemma 5.4]{Card_Lehalle}; we proceed by contradiction assuming that there exists $\e>0$ such that for every $n\in\N\setminus\{0\}$ there exist $\m_n\in C([0,T],\cP(\T^d\times A))$, $\r_n\in(0,1)$ and $t_n\in [0,T-h_n]$, with $h_n\in(0,1/n)$ and
\begin{equation}\label{claim54}
\|Dv_n-Dw_n\|_{L^\infty}\geq\e
\end{equation}
where $v_n$ and $w_n$ are the solutions to~\eqref{HJB_discount} with $\r$ replaced by $\r_n$ and with $(t,\m)$ replaced by $(t_n,\m_n)$ and respectively by $(t_n+h_n,\m_n)$.
Possibly passing to a subsequence, we may assume that the sequence $\{\r_n\}$ converges to some value $\r\in[0,1]$. We split the rest of the proof according to the fact that $\r=0$ or $\r\ne 0$.\\
Case $\r\ne0$.
Estimates~\eqref{eq:bound_1} and~\eqref{eq:bound_2}
and $\r>0$ ensure:
\begin{equation}\label{A9bis}\|v_n\|_{C^{2,\theta}(\T^d)},\|w_n\|_{C^{2,\theta}(\T^d)}\leq \bar K.
\end{equation}
Hence, possibly passing to subsequences (that we still denote $v_n$ and $w_n$), we may assume that $v_n$ and $w_n$ converge to some function $\varphi_v$ and respectively $\varphi_w$ in the topology of $C^{1}(\T^d)$.
	
Since $h_n\to 0$ as $n\to\infty$, possibly passing to a subsequence and without any loss of generality, we assume that both $\{t_n\}_n$ and $\{t_n+h_n\}_n$ converge to a common value~$\bar t$.
By hypothesis $(H5)$, there exists $\bar H(x,p,t)$ such that, as $n\to\infty$, $H(\cdot,\cdot;\cdot, \m_n)$ uniformly converges to $\bar H$ in $\T^d\times B(0,\bar K)\times [0,T]$.
In particular, by Ascoli-Arzela theorem, we deduce that $\bar H$ is continuous on $\T^d\times B(0,\bar K)\times [0,T]$ and, exploiting $(H1')$, also that it satisfies
\begin{equation*}
|\bar H(x_1,p_1,t)-\bar H(x_2,p_2,t)|\leq L \bar K|x_1-x_2|+K|p_1-p_2|
\end{equation*}
for any $x_1,x_2\in\T^d$, $p_1,p_2\in B_{\bar K}$ and $t\in[0,T]$.
We denote by $\bar u$ the unique bounded solution to 
\begin{equation}\label{HJB_bar}
- \D \bar u + \bar H(x,D\bar u,\bar t)+\r \bar u =0 \quad \mbox{ in } \T^d.
\end{equation}
By  $(H5)$ and the continuity of $\bar H$, for some sequence $o_n(1)$ with $\lim_{n}o_n(1)=0$, there holds
\begin{multline}\label{A7bis}
|H(x,p;t_n,\m_n)-\bar H(x,p,\bar t)|\leq |H(x,p;t_n,\m_n)-\bar H(x,p, t_n)|\\+|\bar H(x,p,t_n)-\bar H(x,p,\bar t)|\leq o_n(1)
\end{multline}
for  every $(x,p)\in\T^d\times B_{\bar K}$.
By the Comparison Principle (using the positivity of~$\r$, the bound~\eqref{A9bis}, and the last inequality), we deduce that $v_n\pm o_n(1)$ are super- and subsolution to~\eqref{HJB_bar}. Letting $n\to\infty$, by uniqueness of the solution to~\eqref{HJB_bar}, we obtain $\varphi_v=\bar u$. Repeating the same arguments for $w_n$, we obtain $\varphi_w=\bar u=\varphi_v$. In conclusion, as $n\to\infty$, we have
\[
\|Dv_n(\cdot)-Dw_n(\cdot)\|_{L^\infty}\leq \|Dv_n(\cdot)-D\bar u(\cdot)\|_{L^\infty}+\|Dw_n(\cdot)-D\bar u(\cdot)\|_{L^\infty} \to 0
\]
which contradicts our claim~\eqref{claim54}.

Case $\r=0$. We introduce the functions $v^*_n(\cdot):=v_n(\cdot)-v_n(0)$ and $w^*_n(\cdot):=w_n(\cdot)-w_n(0)$.
Again by estimates~\eqref{eq:bound_1} and~\eqref{eq:bound_2}, we infer
\begin{equation*}
\|v^*_n\|_{C^{2,\theta}(\T^d)},\|w^*_n\|_{C^{2,\theta}(\T^d)}\leq \bar K\qquad\textrm{and}\qquad
|\r_nv_n(0)|, |\r_nw_n(0)|\leq K.
\end{equation*}
By Ascoli-Arzela theorem, possibly passing to a subsequence, there exist a function $V\in C^{2,\theta}(\T^d)$ and a constant $\l_v$ such that
\[
\lim_{n\to\infty}\|v^*_n-V\|_{C^{2}(\T^d)}=0\qquad \textrm{and }
\lim_{n\to\infty}\r_nv_n(0)=\l_v.
\]
By the same arguments as before, relation~\eqref{A7bis} still holds true. Hence, by stability result, we deduce that the function~$V$ solves
\begin{equation}\label{A10}
\l_v-\Delta V+\bar H(x,DV,\bar t)=0,\qquad V(0)=0.
\end{equation}
By similar arguments, there exist a function $W\in C^{2,\theta}(\T^d)$ and a constant $\l_w$ such that
\[
\lim_{n\to\infty}\|w^*_n-W\|_{C^{2}(\T^d)}=0\qquad \textrm{and }
\lim_{n\to\infty}\r_nw_n(0)=\l_w
\]
and consequently
\begin{equation}\label{A10bis}
\l_w-\Delta W+\bar H(x,DW,\bar t)=0,\qquad W(0)=0.
\end{equation}
By~\eqref{A10} and~\eqref{A10bis}, the couples~$(\l_v,V)$ and~$(\l_w,W)$ are both solution to the ergodic problem
\[
\l-\Delta u+\bar H(x,Du,\bar t)=0,\qquad u(0)=0.
\]
By the same arguments as those used in the proof of~\cite[Thm4.1]{AB}, this ergodic problem admits exactly one solution $(\l,u)\in\R\times C(\T^d)$; hence, we have
\[
\l_v=\l_w\qquad \textrm{and }V =W.
\]
Finally, as $n\to \infty$, we conclude
\[
\|Dv_n-Dw_n\|_{L^\infty}=\|Dv^*_n-Dw^*_n\|_{L^\infty}\leq \|Dv^*_n-DV\|_{L^\infty}+\|Dw^*_n-DV\|_{L^\infty}\to 0
\]
which contradicts our claim~\eqref{claim54}.
\end{proof}

\paragraph*{Acknowledgments.} The authors were partially supported by INdAM-GNAMPA Project, codice CUP\textunderscore{E55F22000270001}. The second author was partially supported also by the Fondazione CaRiPaRo Project ``Nonlinear Partial Differential Equations: Asymptotic Problems and Mean-Field Games''.\\
The authors warmly thank the anonymous referees for useful suggestions and comments to improve the paper.

%%%%%%%%%%%%%%%%%%%%%%%%%%%%%%%
%       Bibliografia          %
%%%%%%%%%%%%%%%%%%%%%%%%%%%%%%%


\begin{thebibliography}{99}
	

\bibitem{AB}
Alvarez, O.; Bardi, M. {\it Ergodicity, stabilization, and singular perturbations for Bellman-Isaacs equation.}
Mem. Am. Math. Soc. 204  (2010), no. 960, vi+77 pp.


\bibitem{Bongini_Salvarani}
Bongini, M.; Salvarani, F.
Mean Field Games of Controls with Dirichlet boundary conditions.  arXiv:2111.14209.


\bibitem{camilli_e_altri}
Cacace S.; Camilli, F.; Goffi, A. A policy iteration method for Mean Field Games. ESAIM Control Optim. Calc. Var.,     27 (2021), paper No. 85, 19 pp.

\bibitem{cardaliaguet_notes} 
Cardaliaguet, P. Notes on Mean Field Games: from P.-L. Lions’ lectures at
Coll\`ege de France, Lecture Notes, 2010.

\bibitem{Card_Lehalle}
Cardaliaguet, P.; Lehalle, C.-A. 
Mean field game of controls and an application to trade crowding. 
Math. Financ. Econ. 12 (2018), no. 3, 335-363.

\bibitem{Carmona_Lacker}
Carmona, R.; Lacker, D. A probabilistic weak formulation of mean field games and applications. Ann. Appl. Probab. 25 (2015), no. 3, 1189-1231.

\bibitem{CG19}
Cirant, M.; Goffi, A. On the existence and uniqueness of solutions to time-dependent fractional MFG.  SIAM J. Math. Anal. 51 (2019), no. 2, 913-954.
  
\bibitem{DL}
Dautray, R.; Lions, J.-L.  {\it Mathematical analysis and numerical methods for Science and Technology, vol.5. Evolution problem.} Springer-Verlag, Berlin, 1992.

\bibitem{gomes}
Gomes, D.A.
A stochastic analogue of Aubry-Mather theory.
Nonlinearity 15 (2002), no. 3, 581-603.

\bibitem{gomes_patrizi_voskanyan}
Gomes, D.A.; Patrizi, S.; Voskanyan, V. On the existence of classical solutions for stationary extended mean field games. Nonlinear Anal. 99 (2014), 49-79.

\bibitem{Graber_Mayorga}
Graber, J; Mayorga, S. A note on mean field games of controls with state constraints: existence of mild solutions. arXiv:2109.11655

\bibitem{Kobeissi1}
Kobeissi, Z. On classical solutions to the Mean Field Game system of controls, Comm. Partial Differential Equations 47 (2022), no. 3, 453-488.

\bibitem{Kobeissi2}
Kobeissi, Z.
Mean Field Games with monotonous interactions through the law of states and controls of the agents, arXiv:2006.12949.

\bibitem{lauriere_tangpi}
Lauri\`{e}re, M.; Tangpi, L. Convergence of large population games to mean field games with interaction through the controls. arXiv:2004.08351.

\bibitem{marchi}
Marchi, C. Continuous dependence estimates for the ergodic problem of Bellman equation with an application to the rate of convergence for the homogenization problem. Calc. Var. Partial Differential Equations 51 (2014), no. 3-4, 539-553.

\bibitem{metafune_altri}
Metafune, G.; Pallara, D.;  Rhandi A. Global properties of transition probabilities of singular diffusions.
Teor. Veroyatn. Primen., 54 (2009), 116--148.

\bibitem{mouzouni} 
Mouzouni, C. 
On quasi-stationary mean field games models. Appl. Math. Optim. 81 (2020), no. 3, 655-684.


\end{thebibliography}
\end{document}